\documentclass[11pt]{cedram-aif}

\usepackage{graphicx}
\usepackage{color}
\usepackage{transparent}
\usepackage{standard}
\usepackage{amsrefs}
\usepackage{amssymb}
\usepackage{amsmath, amsgen,amsopn,amstext,amsxtra,amscd,amsbsy}
\usepackage{multirow}

\definecolor{jwcol}{rgb}{0,0,0.8}
\newcommand{\jwnew}[1]{#1}

\newcommand*{\Scale}[2][4]{\scalebox{#1}{\ensuremath{#2}}}%
\newcommand{\hlh}{h \log (1/h)}

\newcommand{\hamvf}{\mathsf{H}}

\newcommand{\restrictedto}{\lvert}
\DeclareMathOperator{\dist}{dist}
\DeclareMathOperator{\Id}{Id}
\DeclareMathOperator{\sgn}{sgn}
\DeclareMathOperator{\Diff}{Diff}
\newcommand{\ext}{\mathcal{E}}
\newcommand{\res}{\mathcal{R}}

\newcommand{\len}{\mathcal{\ell}}
\newcommand{\bs}{\mathbf{s}}
\newcommand{\br}{\mathbf{r}}
\newcommand{\cutoff}{\chi}
\newcommand{\bmD}{\mathbf{D}}
\newcommand{\loc}{\text{loc}}
\newcommand{\ta}{\widetilde{a}}

\renewcommand{\leq}{\leqslant}
\renewcommand{\geq}{\geqslant}
\newcommand{\eps}{\varepsilon}
\newcommand{\Vcal}{\mathcal{V}}
\newcommand{\Mcal}{\mathcal{M}}
\newcommand{\hsh}{\frac{h}{\sqrt{z_h}}}
\newcommand{\zsh}{\frac{\sqrt{z_h}}{h}}
\newcommand{\Pit}{\widetilde{\Pi}}

\thanks{{The authors are grateful to Colin Guillarmou, Nicolas Burq,
  and Maciej Zworski
  for helpful conversations, and to an anonymous referee for
  comments on the manuscript.} The second author thanks Institut Henri
  Poincar\'e and Universit\'e de Paris Sud for their hospitality.  The
  The
  work of L.H. is partially supported by the ANR 
program ANR-13-BS01-0007-01 GERASIC and J.W. acknowledges 
support from NSF grants DMS--1265568 and DMS--1600023.}

\title[Resonances generated by cone points]{On resonances
  generated by conic diffraction}
\author{Luc Hillairet}
\author{Jared Wunsch}
\date{\today}
\begin{document}
\begin{abstract}
We describe the resonances closest to the real axis generated by diffraction of waves among cone
points on a manifold with Euclidean ends. These resonances lie asymptotically evenly
spaced along a curve of the form $$\frac{\Im \lambda}{\log \left |\Re
  \lambda\right |}= -\nu;$$ here
$\nu=(n-1)/2 L_0$ where $n$ is the dimension and $L_0$ is the length of
the longest geodesic connecting two cone points.  Moreover there
are asymptotically no resonances below this curve and above the curve
$$
\frac{\Im \lambda}{\log \left |\Re
  \lambda\right |}= -\Lambda
$$
for a fixed $\Lambda>\nu.$
\end{abstract}
\maketitle
\section{Introduction}
Let $X^n$ be a manifold with cone points $Y_1,\dots Y_N$ and with
Euclidean ends.  We make the geometric assumption that there are no
trapped geodesics that do not hit the cone points, and that there are
a finite number of geodesics $\gamma_{ij}^k$ connecting cone points
$Y_i$ and $Y_j$ for each $i,j$ (with possibly more than one geodesic
connecting a given pair, hence the index $k$).  We further assume that
no two endpoints of any pair of these geodesics $y,y'$ are
$\pi$-related; loosely speaking, this means that no three cone points
are collinear (see \S\ref{section:geometry} below for precise
definitions).

Let $L_0$ denote the longest geodesic connecting two (not a priori
distinct) cone points.  Assume that no two distinct oriented geodesics between cone points
of maximal length $L_0$ end at a common cone point\footnote{Note that this
  hypothesis rules out a geodesic loop of length $L_0$ from a cone
  point to itself, since the geodesic and its reversal end at the same
  point.}  (e.g., one could
assume that there is a unique geodesic of length $L_0$ and it connects
two distinct cone points).  Thus the longest path(s) along two
successive conic geodesics is (are) obtained by traversing back and
forth along a single geodesic, resulting in length $2L_0.$ Let $2L'$
denote the length of the next longest path traversing two successive
geodesics connecting cone points.

The main result of this paper is:
\begin{theorem}\label{theorem:gap}
Let $\lambda_j$ be a sequence of resonances of the Laplacian on a
conic manifold $X$ subject to the geometric hypotheses above, with
$$
-\frac{\Im \lambda_j}{\log \smallabs{\Re \lambda j}} \to \nu.
$$
Then either $\nu=(n-1)/2L_0$ or $\nu\geq \Lambda$ where
$$
\Lambda= \min\{ n/(2L_0), (n-1)/(2L')\}.
$$
More precisely, for a sequence of resonances satisfying
$$
\limsup \bigg(-\frac{\Im \lambda_j}{\log \smallabs{\Re \lambda j}}\bigg)< \Lambda
$$
there exist $\delta>0$ and a constant $C_{\Im}$ such that, {up to extracting a subsequence,}
$$
\Im \lambda_j = -\frac{(n-1)}{2 L_0} \log \smallabs{\Re \lambda_j} +C_{\Im}+ O(\abs{\Re \lambda_j}^{-\delta});
$$
also there exists $C_{\Re}$ such that the following quantization condition holds
\begin{equation}\label{quantization}
\Re \lambda_j \in C_{\Re}+\frac{\pi}{L_0} \ZZ+O(\abs{\Re \lambda_j}^{-\delta}).
\end{equation}
This latter condition should be interpreted as saying the resonances
have real parts in the union of the intervals $B(C_{\Re}+\pi k/L_0,
C \smallabs{k}^{-\delta})$ for $k \in \ZZ$ and for some fixed $C.$

{Moreover, there is only a finite number of pairs $(C_{\Re},C_{\Im})$: one for each geodesic
    of length $L_0$.}
\end{theorem}

The constants $C_{\Re}$ and $C_{\Im}$ are geometrically
meaningful: they are related to the imaginary and real parts
respectively of the logarithm of the product of diffraction
coefficients along {the corresponding} maximal length closed diffractive geodesic
undergoing two diffractions. It also follows from this description
that if the diffraction coefficient along {every maximal} geodesic vanishes then,
for any sequence of resonances  
$$
\limsup \bigg(-\frac{\Im \lambda_j}{\log \smallabs{\Re \lambda_j}}\bigg) \geq  \Lambda.
$$

\jwnew{The theorem applies via the method of images to the Dirichlet or
Neumann problem on the exterior of
one or more polygons in the plane, via a ``doubling'' in which
two copies of the exterior domain are sewn together along their common
edges to make a manifold with cone points---see e.g.\cite[Section 1]{Hillairet:2005},
\cite[Section
1]{BaWu:13}.  As long as no three vertices are collinear, the
collinearity assumption is satisfied; at most one geodesic connects two
cone points;
nontrapping is an open condition (and always satisfied for the exterior of
a single convex polygon); and the longest path condition certainly holds if,
say, no two pairs of vertices are the same distance apart.  Our
geometric hypotheses are thus generically satisfied in the polygonal case (once
nontrapping is stipulated); we expect them to be generic in the more
general nontrapping conic setting as well.}

Our main theorem consists of a upper bound for the locations of resonances of
$X$ lying in a log neighborhood of the real axis (albeit without
multiplicity bounds), implying that \jwnew{resonances in this region
  can only concentrate on a log curve.}  We recall that previous
work of Baskin--Wunsch \cite{BaWu:13}
Galkowski \cite{Ga:15} showed on the one hand \cite{BaWu:13} that
\emph{some} region of the form
$$
-\frac{\Im \lambda_j}{\log \smallabs{\Re \lambda_j}} >\nu_0>0,\quad \abs{\lambda_j}>R
$$
contains no resonances (subject to some genericity conditions on the
relationship among the 
conic singularities); it was more precisely observed by Galkowski \cite{Ga:15} that $\nu_0$
could be taken to be $(n-1)/2L_0+\epsilon.$  On the other hand, the
authors and Galkowski showed\footnote{\jwnew{This theorem was proved by the authors in the
  odd dimensional case and announced by the second author, with a
  sketch of the proof, at the Berkeley/Bonn/Paris Nord/Z\"urich PDE Video
  Seminar in September 2015.  Galkowski subsequently published a proof covering both even
  and odd dimensional cases (citing the authors) as Theorem 5 of \cite{Ga:15}.}}
    the following \emph{existence} theorem for resonances using a trace formula:
\begin{theorem} Assume there is a single maximal orbit of length
  $2L_0$ undergoing two diffractions with nonvanishing diffraction coefficients, whose iterates are all isolated in the length
  spectrum.  Then for
  every $\ep>0,$
$$
\# \big\{\lambda_j: -\frac{\Im \lambda_j}{\log \smallabs{\Re
    \lambda_j}}<\frac{n-1}{2L_0}+\ep\big\}\cap B(0,r) >C_\ep r^{1-\ep}.
$$
\end{theorem}
The proof employs a trace formula of Ford--Wunsch \cite{FoWu:17}
(previously proved by
the first author \cite{Hillairet:2005} in the case of flat surfaces)
describing the singularities of the wave trace induced by diffractive
closed orbits, together with a theorem relating resonances to the
renormalized wave trace due to Bardos-Guillot-Ralston \cite{Bardos-Guillot-Ralston1}, Melrose
\cite{MR83j:35128}, and Sj\"ostrand-Zworski \cite{Sjostrand-Zworski5}
(as well as \cite{Zw:99} for the even dimensional case). It also uses
a Tauberian theorem of Sj\"ostrand--Zworski \cite{Sjostrand-Zworski3}.
  As a proof has appeared in \cite{Ga:15} (see also \cite{Wu:16}) we
  do not give one here.
 
Thus, previous results implied that infinitely many resonances lie
in any logarithmic ``strip''
$$\abs{\frac{\Im \lambda}{\log \smallabs{\Re
    \lambda}}+\frac{n-1}{2L_0}}<\ep.
$$
The results at hand sharpen this result by pushing down further into
the complex plane: we now know that below this first (approximate)
logarithmic curve of resonances there is a gap region.

\jwnew{All this is in marked contrast with the case of a
  non-trapping smooth manifold with Euclidean ends, where classic
  results of Lax-Phillips \cite{Lax-Phillips1} and Vainberg
\cite{Vainberg:Exterior,Vainberg:Asymptotic} show that for for every
$\nu>0,$ there exists $R>0$ so that the region
$$
\Im \lambda>-\nu \log \smallabs{\Re\lambda},\quad \smallabs{\lambda}>R
$$
contains no resonances at all.  The existence of resonances along log
curves is thus a consequence of the \emph{weak trapping} effects of
repeated diffraction at the cone points (see discussion below).  Our results therefore occupy a middle
ground between the smooth nontrapping case and the case of a smooth
manifold with trapped geodesics, where no matter how unstable the
structure of the trapped set, there seem to be resonances
closer to the real axis than those studied here: for instance there are now
numerous results about the existence of resonances lying near lines
parallel to the real axis, generated by normally hyperbolic
trapping---cf.\ \cite{Dy:16}.  (That sequences of resonances
should \emph{always} exist in some strip near the real axis in cases
of trapping on a smooth manifold with Euclidean ends is the content of
the modified Lax--Phillips conjecture.)

Previous results on
strings of resonances on log curves as in Theorem~\ref{theorem:gap} include the
much more precise study in \cite{Burq:Coin} of the related special
case of one orbit bouncing back and forth between an analytic corner and a
wall. The treatise \cite{BoFuRaZe:16} contains similar (and highly
refined) results in the setting of resonances generated by homoclinic
orbits; such resonances are somewhat closer to the real axis than those
discussed here, but the variable-order propagation of singularities
techniques used below also appear in \cite{BoFuRaZe:16}.

The appearance of the factor $(n-1)/(2L_0)$ in our main theorem is quite
natural from a dynamical point of view; after a semiclassical
rescaling of the problem, we show that it represents the minimal
``rate of smoothing'' enjoyed by a solution to the semiclassical
Schr\"odinger equation owing to its diffraction by cone points.
More precisely, consider a putative semiclassical resonant state
$u_h$---this would be a solution to the Schr\"odinger equation with a
complex spectral parameter, here with imaginary part approximately
$-2 \nu h \log(1/h)$.  As we discuss below, the semiclassical
wavefront set for such a solution to the stationary Schr\"odinger
equation propagates along geodesics that are permitted to branch at
cone points. In evaluating the regularity of $u_h$, we show
(Proposition~\ref{proposition:variablepropagation}) that it
\emph{loses} regularity along the forward bicharacteristic flow at a
constant rate proportional to $\nu$.  At each diffraction, by
contrast, $u_h$ generically \emph{gains} regularity by approximately
the factor $h^{(n-1)/2}.$ These gains and losses of regularity must be
in balance along any closed branched geodesic in the wave-front set of
$u_h$.  The smallest $\nu$ will thus be obtained when the diffractive gain in
semiclassical regularity along the branched orbit
is as small as possible per unit length.  We thus show that the optimal scenario is that of
concentration along the closed branching orbit that diffracts as
infrequently as possible: this is the orbit traveling back and forth between
the two maximally separated cone points.  Correspondingly there is a
long-living resonant state concentrated along this orbit that loses energy to infinity via
diffraction as infrequently as possible. It is an instance of the
``weak trapping'' phenomenon referred to above and yields the value
$\nu=(n-1)/(2L_0)$.

It is a natural conjecture that (at least generically), \emph{all} resonances in any log
neighborhood of the real axis lie on quantized log curves $\Im \lambda
\sim -\nu_j \log \smallabs{\Re\lambda}$ for some family of $\nu_j.$
We have been unable to gain sufficient control on error terms to
verify this, however.}

\section{Conic geometry}\label{section:geometry}

We now specify our geometric hypotheses, which are much the same as
those employed in \cite{BaWu:13}: we assume that our manifold has
conic singularities and Euclidean ends, as follows.

Let $X$ be a noncompact manifold with boundary, $K$ a compact subset
of $X$, and let $g$ be a Riemannian metric on $X^\circ$ such that
$X\backslash K$ is isometric to a union of finitely many exteriors of Euclidean balls
$\bigsqcup_j ( \RR_{z_j}^n\backslash \overline{B^n(0,R_0)})$ and such that $g$ has conic
singularities at the boundary of $X$:
$$
g=dr^2+ r^2 h(r,dr,y,dy);
$$
here $g$ is assumed to be nondegenerate over $X^\circ$ and
$h |_{\pa X}$ induces a metric on $\pa X,$ while $r$ is a boundary
defining function. We let $Y_\alpha,$ $\alpha=1,\dots N$ denote the
components of $\pa X;$ we will refer to these components in what
follows as \emph{cone points}, as each boundary component is a single
point when viewed in terms of metric geometry.

For simplicity of notation, we will retain the notation $B^n(0, R)$
(with $R \gg 0$) for
the union of $K$ and the intersection(s) of this large ball with the
Euclidean end(s) $X \backslash K$.

Theorem~1.2 of \cite{Melrose-Wunsch1} implies that we may
choose local coordinates $(r,y)$ in a collar neighborhood of each $Y_\alpha$
such that the metric takes the form
\begin{equation}\label{semiproduct}
g=dr^2+ r^2 h(r,y,dy),
\end{equation}
where $h$ is now a family in $r$ (which is the distance function to the boundary)
of smooth metrics on $Y_\alpha.$  The curves $\{r=r_0\pm t,\ y=y_0\}$
are now unit-speed geodesics hitting the boundary, and indeed are the
\emph{only} such geodesics.

We will say that the conic manifold $X$ is of \emph{product type} if
locally near $\pa X$ the metric can be written in the form
\begin{equation}\label{productmetric}
g=dr^2 + r^2 h(y,dy)
\end{equation}
in some product coordinates in a collar neighborhood of $\pa X,$ where $r$ is a defining function and
where $h$ is a \emph{fixed} (i.e., $r$-independent) metric on $\pa X.$

We say that the concatenation of a geodesic entering the
boundary at $y=y_0\in Y_\alpha$ and another leaving at $y=y_1 \in
Y_\alpha$ \emph{at the same time} is a \emph{geometric} geodesic if
$y_0, y_1$ are connected by a geodesic in $Y_\alpha$ (with respect to
the metric $h\restrictedto_{r=0}$) of length $\pi.$ Such
concatenations of geodesics turn out to be exactly those which are
locally approximable by families of geodesics in $X^\circ$ (see
\cite{Melrose-Wunsch1}).  In the special case of a surface with
conic singularity there are locally just two of these, corresponding to
limits of families of geodesics that brush past the cone point on
either side; more generally, there is a (locally) codimension-two family of such
geodesics through any cone point.

By contrast, we say 
that concatenation of a geodesic entering the
boundary at $y=y_0\in Y_\alpha$ and another leaving at $y=y_1 \in
Y_\alpha$ \emph{at the same time} is a \emph{diffractive} geodesic if
there is no restriction on $y_0,y_1$ besides lying in the same
boundary component $Y_\alpha.$  We say that a diffractive geodesic is \emph{strictly
  diffractive} if it is not geometric.  We say that two points in the
cotangent bundle of $\RR\times X$ are \emph{diffractively related}
(resp.\ strictly diffractively, geometrically) if
they are connected by a diffractive geodesic (resp.\ strictly
diffractive, geometric). 

The principal results of \cite{Melrose-Wunsch1} (see also
\cites{Cheeger-Taylor1,Cheeger-Taylor2}) are that singularities for
solutions of the wave equation propagate along diffractive geodesics,
with the singularities arising at \emph{strictly} diffractive
geodesics being generically weaker than the main singularities.  More
precisely, if $q$ is a point  with
coordinates $(r_0,y_0)$ lying close to a cone point $Y_\alpha,$ the
solution
$$
u\equiv e^{-it \sqrt{\Lap}} \delta_q
$$
is shown to have a \emph{conormal} singularity at the diffracted
wavefront $r=t-r_0$ (for $t>r_0$) lying in $H^{-1/2-\epsilon}$ away
from the geometric continuations of the geodesic from $q$ to $Y_\alpha$
(whereas the main singularity is in $H^{-n/2-\epsilon}$). The symbol
(and also precise order) of this singularity was analyzed in
\cite{FoWu:17}, based on computations in the product case by
Cheeger--Taylor \cites{Cheeger-Taylor1,Cheeger-Taylor2}, yielding the following:
\begin{prop}[\cite{FoWu:17}]\label{proposition:diffractivepropagator}
 Let $p = (r,y)$ and $p' = (r',y')$ be strictly diffractively related
 points near cone point $Y_\alpha.$  Then near
  $(t,p,p')$, the Schwartz
  kernel of the half-wave propagator $e^{-it\sqrt\Lap}$ (acting on
  half-densities) has an oscillatory
  integral representation
  \begin{equation}\label{FWpropagator}
 e^{-it\sqrt{\Lap}} = \int_{\RR_\xi} e^{i(r + r' - t) \cdot \xi} \,
    a_D(t,r,y;r',y';\xi) \, d\xi
  \end{equation}
  whose amplitude $a_D \in S^0$ is
  \begin{equation}
    \label{eq:nonproduct-hw-amp}
    \frac{(rr')^{-\frac{n-1}{2}}}{2\pi i} \, \cutoff(\xi) \cdot
    \bmD_\alpha(y,y') \cdot
    \Theta^{-\frac{1}{2}}(Y_\alpha \to y) \, \Theta^{-\frac{1}{2}}(y' \to
    Y_\alpha) \, \omega_g(r,y) \, \omega_g(r',y') 
  \end{equation}
  modulo elements of $S^{-\frac{1}{2} + 0}$.  Here,
  $\cutoff \in \mathcal{C}^\infty(\RR_\xi)$ is a smooth function satisfying
  $\cutoff \equiv 1$ for $ \xi > 2$ and $\cutoff \equiv 0$ for $ \xi <
  1$.  The factor $\bmD_\alpha(y,y')$ is the Schwartz kernel of
  the operator $e^{-i\pi\sqrt{\Lap_{Y_\alpha}+(n-2)^2/4}},$ while the
  factors $\Theta^{-\frac 12}$ are given by nonvanishing determinants of Jacobi
  fields (cf.\ \cite{FoWu:17} for details).  \\

In the case where $X$
  is of product type near $Y_\alpha,$ the amplitude $a_D$
  admits an asymptotic expansion in powers of
  $\abs{\xi}^{1/2}.$
\end{prop}

\section{Analytic preliminaries}

We begin by making a semiclassical rescaling of our problem. 
Existence of a resonance $\lambda$ implies the existence of a certain
kind of
``outgoing'' solution 
$u$ (an associated \emph{resonant state}) of the equation 
$$
(\Lap-\lambda^2)u=0.
$$
Setting $\Re \lambda =h^{-1}$ and $\Im \lambda = -\nu \log \Re
\lambda=-\nu \log h^{-1}$ gives
\begin{equation}\label{eq:defres}
(h^2\Lap-z_h)u=0
\end{equation}

where
$$
z_h = (1-i\nu\hlh)^2 = 1+O( (\hlh)^2 )-2i \nu \hlh.
$$
Our aim is to understand the resonances of the family $h^2 \Lap-z_h$
in a logarithmic neighbourhood of the real axis. More precisely, we
look for resonances in the set of $z_h$ with
$$\sqrt{z_h} \in \Omega_\ep\equiv \{
  (-\Lambda +\epsilon)\hlh <\Im\sqrt{ z_h}<0,\ \Re \sqrt{z_h} \in [1-\ep,1+\ep] \}$$

Now we recall (cf.\ \cite{Sjostrand-Zworski1} or Chapter 4 of \cite{DyZw:15})
that the resonances of the Laplacian with argument having
magnitude less than a fixed $\theta$ agree with poles of
\emph{complex-scaled} operator $\Lap_\theta$ which coincides with the
original Laplacian on a large compact set but which, near infinity, is
deformed into the complex domain.  Existence of a resonance at $z_h$
is then equivalent to existence of an $L^2$ eigenfunction of the
non-self-adjoint complex scaled problem.

As in \cite{BoFuRaZe:16}*{Section 2}, \cite{Sjostrand-Zworski:Fractal} we
will scale only to an angle $O(\hlh);$ this restriction on the scaling
has the virtue (albeit an inessential one here) that the overall
propagation of singularities near the scaling is still bi-directional:
while propagation into the scaling region loses powers of $h,$ it only
loses a finite number of such powers (see
Proposition~\ref{proposition:variablepropagation} below).  Scaling to
fixed angle, by contrast, would break the propagation of semiclassical
singularities at the boundary of the scaled region.  Thus we fix
$M\gg 0$ and set $\theta$ so that 
\begin{equation}\label{thetadef}\tan
  \theta =M
  \hlh.\end{equation} 
 
We let $D_x= -i \partial_x$ and we consider an operator given by the
Laplace-Beltrami operator $\Lap_g$ on the compact part $K \subset X$
and, on the ends (with Euclidean coordinate $x\in \RR^n$), given by the
expression
$$
\Lap_\theta =\big((I+ i F_\theta''(x))^{-1} D_x\big)\cdot \big((I+
i  F_\theta''(x))^{-1} D_x\big),
$$
where
$$
F_\theta(x) =(\tan \theta)\cdot g(\abs{x})
$$
for a function $g$ chosen so that 
$$
g(t) = 0,\ t\leq R_1, \quad g(t)=\frac 12 t^2,\ t>2R_1,\ g''(t) \geq 0.
$$
We now set
$$
P_\theta=h^2 \Lap_\theta,
$$

By construction, we have the following decomposition of $X$ 
\begin{enumerate}
\item The interior region $B^n(0,R_1)$ in which $P_\theta=h^2\Delta.$
\item The scaling region $R_1\leq R\leq 2R_1$. In this region, we
  have 
$\Id+iF_{\theta}''(x) = \Id+\tan \theta G(x) = \Id+ iMh\log (1/h)\cdot G(x),$
for some $n\times n$ matrix $G$ with smooth entries. 
We then obtain the asymptotic expansion 
\[
P_\theta \sim h^2\Delta \,+\, \sum_{k\geq 1} (\tan \theta)^k Q_k
\]
for some $Q_k \in \Diff^2_h.$
\item The deep scaling region ($R\geq 2R_1$) in which 
\[
P_\theta = (1+i\tan \theta)^{-2}h^2 \Delta.
\]
\end{enumerate}
This decomposition implies that there exists a constant $C,$
independent of $h,$ such that, for any $v$ 
\begin{equation}\label{eq:estscal}
\| (P_\theta-h^2\Delta) v\|_{L^2} \,\leq\, C \hlh \| (h^2\Delta +1) v\|_{L^2} 
\end{equation}
 
Let us assume as above, that, for a sequence of resonances, there
exists $\nu, E \in (0,\infty)$ 
for which 
$$
z_h=E+o(1)-i (2\nu+o(1)) \hlh.
$$
Since existence of resonances is equivalent to existence of eigenvalues of the complex-scaled operator
$P_\theta$, there exists a sequence of solutions to
\begin{equation}\label{eigenfunction}
(P_\theta-z_h) u_h=0,
\end{equation}
that is normalized in $L^2.$

There is a slight ambiguity in the nomenclature since the eigenfunction $u_h$
in the latter equation differs from the resonant state in
(\ref{eq:defres}), although the $z_h$ are the same. We will say that
this $u_h$ is a resonant state. 
 
The semiclassical principal symbol of $P_\theta-z_h$ is then
given by
$$
\sigma_h(h^2\Lap_\theta)-E+o(1)+ i (2\nu+o(1))\hlh .
$$
We easily see as in \cite{Sjostrand-Zworski1}, \cite{DyZw:15} that for
$h$ sufficiently small,
$$
-C_1 \theta \abs{\xi}^2\leq  \Im \sigma_h(h^2\Lap_\theta) \leq 0
$$
and 
$$
 \Im \sigma_h(h^2\Lap_\theta) \leq -C_2 \theta \abs{\xi}^2,\ \abs{x}>2 R_1.
$$
Hence if $M$ is sufficiently large (relative to $\nu$), $P_\theta-z_h$
enjoys a kind of  semiclassical hypoellipticity in $\abs{x}>2R_1$
in the sense that as $h \downarrow 0$ the imaginary part of its symbol is a nonvanishing
multiple of $\hlh$ in this region.  This will have consequences 
for regularity of solutions to $(P_\theta-z_h)u_h=O(h^\infty)$ that we
will derive in the following sections.

\section{Semiclassical wavefront set and propagation of singularities}

Let $(u_h)_{h\geq 0}$ be a bounded sequence in $L^2(X).$ For a
positive sequence $\epsilon_h$, we will use
the notation $u_h=O_{L^2}(\epsilon_h)$ to say that
$\epsilon_h^{-1}u_h$ is a bounded sequence in $L^2(X).$

For $(x,\xi) \in T^*X^\circ$ we define (cf.\ \cite[Chapter 8]{Zw:12})
$$
(x,\xi) \notin \WF_h^s u_h \Longleftrightarrow \text{there exists } A
\in \Psi_h(X) \text{ elliptic at } (x,\xi) \text{ and } A u_h=O_{L^2}(h^s).
$$
Likewise
$$
(x,\xi) \notin \WF_h u_h \Longleftrightarrow \text{there exists } A
\in \Psi_h(X) \text{ elliptic at } (x,\xi) \text{ and } A u_h
= O_{L^2}(h^\infty).
$$
As usual, we have
$$
\WF_h u_h =\overline{\bigcup_s \WF_h^s u_h}.
$$
(Note that, for the moment, we are only dealing with wavefront set
\emph{away} from cone points)

By standard elliptic regularity in the semiclassical calculus, we have
the following result.
\begin{lemma}\label{lemma:ellipticity}
Let $(P_\theta-z_h)u_h=O_{L^2}(h^\infty)$ with $\Re z_h=E+o(1).$

We have
$$
\WF_h u_h\cap T^*X^\circ \subset \{\abs{\xi}_g^2=E\}.
$$
\end{lemma}

This lemma implies that the wave-front set of $u_h$ does not
intersect the $0-$section in $T^*X_0.$ Thus, we may test $u_h$ against
standard (non-semiclassical) pseudodifferential or Fourier Integral
operators in order to understand its semiclassical wave-front. To
understand 
the propagation of singularities, it
thus suffices to study $U(t)u_h,$ where $U(t)$ denotes the half-wave
propagator $U(t) =\exp(-it\sqrt{\Lap}).$

The rest of this section is devoted to the proof of the following propagation of
singularities result. What is novel here is the variable semiclassical order (in addition to the presence
of complex scaling).  Note also that this proposition only governs propagation in
$X^\circ,$ away from cone points.  

\begin{prop}\label{proposition:variablepropagation}
Let
$$(P_\theta-z_h)u_h=O_{L^2}(h^\infty)$$ 
with
$$
z_h=E+o(1)-i(2 \nu+o(1))\hlh.
$$
\begin{enumerate}\item\label{prop:fwd}
For $t>0$ assume that
$\{\exp_{t' \hamvf}(q),\ t' \in [0, t]\}\subset T^* (X^\circ \cap B(0,R_1)).$ 
Then $$
q \notin \WF_h^s u_h\Longrightarrow \exp_{t \hamvf} (q) \notin \WF_h^{s'} u_h$$
for all $s'<s- 2t\nu.$
\item \label{prop:back}
For $t>0$, assume that
$\{\exp_{t' \hamvf}(q),\ t' \in [0, t]\}\subset T^*(X^\circ \cap B(0,R_1)).$ 
Then
$$
\exp_{t \hamvf} (q) \notin \WF_h^{s'} u_h\Longrightarrow q \notin \WF_h^s u_h
$$
for all $s<s'+2t\nu,$
\item \label{prop:fwdback}
There exists some $M_0>0$ such that for \emph{any}
  $q \in T^*X^\circ$ and $t \in \RR$ (not necessarily positive)
$$
q \notin \WF_h^{s} u_h
\Longrightarrow 
\exp_{t \hamvf} (q) \notin \WF_h^{s'} u_h,
$$
for all $s'<s-M_0\smallabs{t}$
provided $\{\exp_{t' \hamvf} (q): t'\text{ between } 0 \text{ and } t\} \subset T^*X^\circ.$
In particular this implies
\begin{equation}\label{simplepropagation}
q \in \WF_h u_h \Longleftrightarrow
\exp_{t \hamvf} (q) \in \WF_h u_h,
\end{equation}
provided the flow does not reach a cone point.
\item\label{prop:inscaling}
There exists $C>0$ such that for $t>0$ and
$\{\exp_{t' \hamvf}(q),\ t' \in
[0, t]\}\subset T^*(X\backslash B(0, 2R_1)),$ 
$$
q\notin  \WF_h^{s} u_h \Longrightarrow  \exp_{t \hamvf} (q) \notin \WF_h^{s'} u_h
$$
for all $s'<s+Ct.$
\end{enumerate}
\end{prop}

\begin{rema}
  The main content of the proposition is that, in the interior region
  and under forward propagation, semiclassical regularity drops at a rate $2\nu t$ on
  the $t$-parametrized flow along $\hamvf,$ the Hamilton vector field
  of the symbol of $P.$ Note, though, that the vector field $\hamvf/2\sqrt{E}$
  induces unit speed geodesic flow, hence the rate of regularity loss along
  unit-speed geodesics is $\nu t/\sqrt{E},$ where $E$ is the real part
  of the spectral parameter and $-2i\nu\hlh$ the imaginary part.

The situation is more complicated near the boundary of the scaling
region, where there are gains or losses in regularity owing to the
scaling which compete with the $2\nu t$ loss rate, but these gains and losses are also of finite order.
The last
part illustrates that deep within the scaling region, forward propagation
\emph{gains} regularity; this will be our substitute for elliptic regularity
in the scaling region, since we do not have full semiclassical
elliptic regularity with the angle $\hlh$ scaling employed here:
semiclassical singularities instead propagate but decay.
(Cf.\ \cite{BoFuRaZe:16}*{Lemma 8.4} for related results.) 
\end{rema}

\begin{rema}
This proposition can be proved using commutator arguments involving
operators of variable
semiclassical order.  We refer the reader to \cite[Section 2.3]{Ga:14}
for a thorough treatment of the symbols of these operators.  The main
difference with ordinary commutator arguments is the presence of $\log
(1/h)$ losses in the computation of the Poisson bracket (see also 
\cite[Appendix A]{BaVaWu:15} for analogous discussion in the
homogeneous setting). We have chosen a different approach using the
half-wave propagator since we will need to understand $U(t)u_h$ anyway
to go through the conical points. 
\end{rema}

In the next subsection, we prove parts \eqref{prop:fwd} and
\eqref{prop:back} of Proposition \ref{proposition:variablepropagation}
and in the following one we will turn to the parts \eqref{prop:fwdback} and \eqref{prop:inscaling}.   

\subsection{Propagation in the interior region}
Consider a resonant state $u_h,$ i.e., a solution to 
$$
(P_\theta-z_h) u_h=0.
$$
Then of course locally in $B(0, R_1)$ we have simply
$$
(h^2\Delta-z_h) u_h=0
$$
so that, by finite speed of propagation,
\[
\frac{\sin( t\sqrt{\Delta})}{\sqrt{\Lap}} u_h \,=\, \frac{h\sin(t\frac{\sqrt{z_h}}{h})}{\sqrt{z_h}} u_h, 
\]
in $B(0,R_1-A)$ for any $|t|<A.$ 

We need a little more work to understand $U(t)u_h.$

\begin{lemma}\label{lem:Utu}
Let $u_h$ be a resonant state associated with the resonance $z_h$ with
$\sqrt{z_h} \in \Omega_\ep.$
For any $|t|<A$ and any open set $V$ such that $\overline{V}\subset
B(0,R_1-A)\cap X^\circ.$

\begin{equation}\label{reproducing}
U(t) u_h = e^{-it \sqrt{z_h}/h} u_h+O(h^\infty)
\end{equation}
in $L^2(V).$
\end{lemma}

\begin{proof}
Since $(P_\theta-z_h)u_h=0,$ we have $(h\sqrt{\Lap}+\sqrt{z_h}) (h\sqrt{\Lap}-\sqrt{z_h})
u_h=0$ in $V.$ Using the results of Appendix \ref{appendix:micro}, we know
that, in $V,$ $\sqrt{\Lap}$ is a
pseudodifferential operator with real symbol so that, by ellipticity of the first factor, we obtain
$(h\sqrt{\Lap}-\sqrt{z_h})u_h=O(h^\infty)$  in $L^2(V).$ 
We now set 
\[
v_h(t,\cdot)\,=\,\frac{\sin( t\sqrt{\Delta})}{\sqrt{\Lap}} u_h.
\]
We have 
\[
\begin{split}
U(t) u_h & = \partial_t v_h(t,\cdot)\,-\, i\sqrt{\Lap} v_h \\
&= \cos( t\frac{\sqrt{z_h}}{h})u_h\,-\,i
\frac{h\sin(t\frac{\sqrt{z_h}}{h})}{\sqrt{z_h}}\sqrt{\Lap} u_h.
\end{split}
\]
The claim follows since $h\sqrt{\Lap} u_h = \sqrt{z_h}u_h +O(h^\infty).$
\end{proof}

\begin{proof}[Proof of \eqref{prop:fwd} and
\eqref{prop:back} of Proposition \ref{proposition:variablepropagation}]
Let $(x_0,\xi_0)$ in $T^*X^\circ\cap B(0,R_1)$ and $t$ such that the
geodesic of length $t$ emanating from $(x_0,\xi_0)$ stays in
$T^*X^\circ\cap B(0,R_1).$ We write $\phi_t$ for  $\exp_{t \hamvf}.$

If $(x_0,\xi_0) \notin \WF_h^s(u_h)$ then we can find microlocal
cutoffs $\Pi_0$ near $(x_0,\xi_0)$ and $\Pi$ near $\phi_t(x_0,\xi_0)$ such
that $\Pi_0 u_h = O_{L^2}( h^{s})$ and $\Pi U(t) (I-\Pi_0)$ is
smoothing. It follows that 

\[
\begin{split}
\Pi u_h &=\, e^{it \sqrt{z_h}/h}\Pi U(t) u_h \,+\, O_{L^2}(h^\infty)\\
& =\,  e^{it \sqrt{z_h}/h}\Pi U(t) \Pi_0 u_h \,+\,O_{L^2}(h^\infty)
\end{split}
\]

By unitary of $U(t),$
\[
 \| e^{it \sqrt{z_h}/h}\Pi U(t) \Pi_0 u_h \| \,\leq
 \,h^{-2(\nu t+0)}\|\Pi_0 u_h\|,
\] 
hence
\begin{equation}
  \label{eq:2}
  \Pi u_h =O_{L^2}(h^{s-2\nu t-0}).\qed
\end{equation}
\noqed
 \end{proof}

\subsection{Propagation in the scaling region}
We now prove a propagation result which is valid everywhere 
in $T^* X^\circ,$ including the scaling region.

Let $(x_0,\xi_0)\notin \WF_h^s(u_h).$  Assume that {$\phi_t(x_0,\xi_0)$} lies over
$X^\circ$ for all $t \in [0, T],$ for some $T>0.$ We will show that {$\phi_t(x_0,\xi_0) \notin
\WF_h^{s-M_0t-\epsilon} u_h$} for all $t\in [0,T]$ and $\epsilon>0;$ the case of negative values of
$t$ will follow by the exact same argument, which is not sensitive to
choices of sign. The constant $M_0$ will be defined during the proof.

We prove by (descending) induction that for any $N\in \NN$ the
following holds {true:} for any $\eps>0,$ we have 
{$\phi_t(x_0,\xi_0) \notin \WF_h^{s-N-M_0t-\epsilon}$} for $t \in [0,T].$ 

This assertion certainly holds true if $N$ is large enough (since $u_h \in L^2$).
We now assume that it holds true for some $N$, set $s'=s-N,$ and
choose $\epsilon>0.$ 

There exists $\Pi_0 \in \Psi_h(X^\circ),$ elliptic at $(x_0, \xi_0),$
with $\Pi_0 u=O_{L^2}(h^s).$ Let
$$\Pi(t) = U(t) \Pi_0 U(-t);$$ by the results of
Appendix~\ref{appendix:micro} this is in fact a pseudodifferential
operator elliptic at $\phi_t(x_0, \xi_0).$  By our
inductive hypothesis, we may shrink $\Pi_0$ if necessary (but still
include a \emph{fixed} open neighborhood of $(x_0, \xi_0)$ in its elliptic set) and then
additionally find $\Pit(t)$ elliptic on $\WF'\Pi(t)$ with
$\Pit(t) u=O_{L^2}(h^{s'-M_0t-\frac{\epsilon}{2}}).$

Microlocally in $T^*X^\circ$, the operator $h\sqrt{\Delta}+\sqrt{z_h}$ is an elliptic pseudodifferential
operator so that  
\[
\sqrt{\Delta} u_h -\frac{\sqrt{z_h}}{h} u_h \,=\,Q_h u_h, 
\]
where we have set 
\[
Q_h := h^{-1} (h\sqrt{\Delta}+\sqrt{z_h})^{-1} \left(
  h^2\Delta-P_\theta\right ).
\]
Thus $Q_h \in \log(1/h) \Psi_h(X^\circ)$ in the scaling region (and
vanishes where the scaling vanishes).

Let $f(t) =\norm{\Pi(t) u_h}.$  Then we compute:
\begin{align*}
2 f f'=\frac{d}{dt} f^2 &= 2 \Im \ang{[\sqrt{\Lap}, \Pi(t)] u_h,\Pi(t)u_h}\\
&= -2 \Im \ang{\Pi(t)\sqrt{\Lap} u_h, \Pi(t) u_h}\\
&= -2 \Im \ang{\Pi(t)(\sqrt{z_h}/h +Q_h) u_h, \Pi(t) u_h}\\
&= -2 \Im\sqrt{z_h}/h \norm{\Pi(t) u_h}^2 -2 \Im \ang{\Pi(t) Q_h u_h, \Pi(t)
  u_h}\\
&= -2 \Im\sqrt{z_h}/h \norm{\Pi(t) u_h}^2 -2 \Im \ang{ Q_h \Pi(t) u_h, \Pi(t)
  u_h}\\
&~~-2 \Im \ang{ [\Pi(t), Q_h] u_h, \Pi(t) u_h}.
\end{align*}
There exists a constant $M_0$ (which is independent of $s'$) such that 
$$M_0>\frac{(-\Im \sqrt{z_h}/h)+\norm{Q_h}}{\log(1/h)},$$ and another
constant $C$ such that $  \| [\Pi(t), Q_h] u_h\| \,\leq\, Ch\log(1/h) 
\smallnorm{\widetilde{\Pi}(t) u_h}.$
Hence we obtain 
\begin{align*}
2 f f '&\leq 2M_0\log(1/h) \norm{\Pi(t) u_h}^2 + C h
  \log(1/h)\smallnorm{\widetilde{\Pi}(t) u_h}\norm{\Pi(t) u_h}\\
&\leq 2M_0\log(1/h) f^2 + C h^{1+s'-M_0t-\frac{\epsilon}{2}}
  \log(1/h) f.
\end{align*}

An application of the Gronwall inequality now yields
$$
f(t) \leq f(0) h^{-M_0t} + C h^{1+s'-M_0t-\epsilon }h^{\frac{\epsilon}{2}}\log (1/h),
$$
hence our assumption on $\Pi_0$ yields
$$
f(t) \leq C h^{1+s'-M_0t-\epsilon},
$$
which, since $\Pi(t)$ is elliptic on $\phi_t(x_0, \xi_0),$
completes the proof of the induction step. 
Part (3) of the proposition follows by setting $N=0$.

Finally, in the deep scaling region, we follow the same argument as in the
un-scaled region, using that 
\[
P_\theta u_h \,=\ z_h u_h \implies \sqrt{\Delta} u_h \,=\,
(1+ih\tan \theta)\frac{\sqrt{z_h}}{h} u_h
\]
so that we obtain, in place of \eqref{eq:2}, the \emph{gain} in regularity
\[
\| \Pi u_h \| =O(h^{s+Ct}).
\] 

\section{Microlocal concentration on the outgoing set}
We now deduce from the preceding section the fact that the wavefront set 
of a resonant state only lives on the outgoing set, which we define
below.  Unless otherwise specified below, we will take the asymptotics
\eqref{z} for $z_h$ as a standing assumption from now on, so that we
may apply the propagation theorems obtained above.

\begin{coro}\label{corollary:nontrapped}
Let
$$(P_\theta-z_h)u_h\,=\,O(h^\infty)$$ 
with
\begin{equation}\label{z}
z_h=E+o(1)-i(2 \nu+o(1))\hlh.
\end{equation}
Let $q \in T^*X^\circ$ and assume that $\exp_{-t \hamvf}(q)\in T^*
X^\circ$ for all $t>0.$  Then $q \notin \WF_h u_h.$
\end{coro}
\begin{proof}
The assumption on the flowout of $q$ means, given our
standing hypothesis that geodesics in $X^\circ$ are non-trapped, that
the backward flowout of $q$ eventually escapes into the deep scaling region
$X\backslash B(0, 2R_1).$  Thus there exists $T_0$ such that for
$t\geq T_0,$ $\pi(\exp_{-t \hamvf} q) \notin B(0, 2 R_1).$ Since
$u_h \in L^2,$ part {\ref{prop:inscaling}} of
Proposition~\ref{proposition:variablepropagation}, applied to a
neighborhood of the backward flowout of $\exp_{-T_0 \hamvf} (q)$ for
arbitrarily long times, gives
$\exp_{-T_0\hamvf} (q) \notin \WF_h u_h.$  Part
\ref{prop:fwdback} then yields $q \notin \WF_h u_h.$
\end{proof}

This proposition tells us that a resonant state $u_h$ can only have
wave-front on rays that emanate from the conical points. 
This leads to the following definition.

\begin{defi} Let $\Gamma_{\pm}$  denote
the flowout/flowin from/to the union of cone points, i.e.\
$$\Gamma_{\pm}=\big\{q \in T^*X^\circ\colon \exp_{t\hamvf}(q) \to Y
\text{ as } t \to t_0^\pm \text{ for some } t_0 \lessgtr 0\big\}
$$
Let $\Gamma
\equiv \Gamma_+\cap \Gamma_-$ denote the geodesics propagating among
the cone points.\end{defi}

Locally near $Y,$ in coordinates $(r,y)$ from
\eqref{semiproduct} with canonical dual coordinates $\xi,\eta,$
we have $\Gamma_\pm=\{\xi\gtrless 0,\ \eta=0\}.$ 
The preceding corollary thus says that 
\[
\WF_h(u_h)\subset \Gamma_+.
\]

The set $\Gamma$ is the trapped set that corresponds to our setting
(cf.\ Chapter 6 in \cite{DyZw:15}). It consists in the geodesic rays
that connect two (not necessarily distinct) conical points.  We now
proceed to show that a non-trivial resonant state must have
\emph{some} wavefront set on $\Gamma.$ This will result from the
composition with the half-wave propagator near a conical point and of
the known structure of this operator.

\section{Composition with the wave propagator}\label{section:composition}
Recall that $U(t)$ denotes the half-wave propagator :
\[
U(t) =e^{-it\sqrt{\Lap}}. 
\]
The propagation results of \cite{Melrose-Wunsch1},
\cites{Cheeger-Taylor1, Cheeger-Taylor2} translate immediately into
statements on propagation of semiclassical
wavefront set:
\begin{prop}\label{proposition:propofsings}
Let $A_h,\ B_h \in \Psi_h(X^\circ)$ be compactly supported semiclassical pseudodifferential operators with
microsupports in a neighborhood of a cone point $Y_\alpha$ that are strictly diffractively related (i.e.\ no
geometric geodesics through $Y_\alpha$ connect a point in $\WF_h'A_h$ and $\WF_h' B_h$)
and with $$(\WF_h' A_h \cup \WF_h' B_h) \cap (0-\text{section})=\emptyset.$$
Then for any $f_h$ with $(P_h-z_h)f_h=O(h^\infty)$ we have
$$
\WF_h( A_h U(t) B_h f_h) \subset D \circ \WF_h f_h
$$
where $D$ denotes the canonical relation in canonical coordinates $(r,\xi,\theta,\eta)$
$$
\{(r,\xi,y,\eta,r',\xi',y',\eta'): r+r'=t, \eta=\eta'=0, \xi=-\xi'\}.
$$
Moreover a quantitative version of this result holds: there exists
$N$ such that
$$
\WF_h^s( A_h U(t) B_h f_h) \subset D \circ \WF_h^{s+N} f_h.
$$
\end{prop}
As a special case, this result tells us that if there is no
wavefront set at all on geodesics arriving at $Y_\alpha$ then there is
no wavefront set on geodesics leaving it.
\begin{proof}
The main result of \cite{Melrose-Wunsch1} is that near diffractively
related points, for fixed $t,$ the Schwartz kernel of $U(t)$ is a conormal
distribution with respect to $r+r'=t.$  Then (8.4.8) of \cite{Zw:12}
shows that in fact we locally have
$$
\WF_h(U(t))\backslash\{0-\text{section}\} = N^*\{r+r'=t\}.
$$
The mapping property on $\WF_h$ then follows from the usual results on
mapping properties of FIOs.

The
quantitative version follows from the closed graph theorem.  (We could
of course get an explicit $N$ but it is immaterial for our purposes.)
\end{proof}

\begin{rema}
The precise form of the principal symbol of the conormal distribution
will not concern us so much as the mere fact of conormality (and the
order of the distribution).
\end{rema}

As a consequence of this lemma and of
Proposition \ref{proposition:propofsings} together with our ``free''
propagation results above
(Proposition \ref{proposition:variablepropagation}), we may now draw the
desired conclusion about the microsupport of a resonant state.

\begin{coro}\label{corollary:propagation}
Suppose $u_h \in L^2(X)$ and
$$
(P_\theta-z_h)u_h=0.
$$
Then $\WF_h u_h\subset \Gamma_+.$  If $\WF_h u_h \cap
\Gamma=\emptyset$ then $u_h=O(h^\infty).$
\end{coro}
\begin{proof}
The containment in $\Gamma_+$ is simply
Corollary~\ref{corollary:nontrapped} above.

Now suppose that there is no wavefront set along $\Gamma=\Gamma_+\cap \Gamma_-.$
Consider one cone point $Y_\alpha$. It is possible to choose, for small $\epsilon$, $B_h$ that microlocalizes near the sphere of radius
$\epsilon$ centered at $Y_\alpha$ in the incoming directions and $A_h$
that microlocalizes near the same sphere but in the outgoing
directions, so that, setting $t=2\epsilon$,
\[
\WF_h A_h U(t) u_h = \WF A_h U(t) B_h u_h.
\]
Applying Proposition \ref{proposition:propofsings}, we find $\WF_h
A_hU(t)B_h u_h=\emptyset.$ The latter equality and equation
\eqref{reproducing} then implies that $\WF_h A_h u_h=\emptyset.$
Since this argument works for any $\alpha,$ we see from our
propagation results above that
$u_h=O(h^\infty)$ globally.
\end{proof}

It is thus sufficient to understand $u_h$ near the rays in $\Gamma$ in
order to understand its global behavior.  In the following section, we
encode the behaviour of $u_h$ near such a ray by restricting $u_h$ to
a transversal cross-section. This is reminiscent of the construction
of the quantum monodromy matrix of Nonnenmacher--Sj\"ostrand--Zworski
(see \cites{NSZ:11,Nonn:12}).

\section{Restriction and extension}\label{section:restriction}




Let $S \subset X^\circ$ be an open oriented hypersurface and let
$(x,y)$ denote normal coordinates near $S,$ i.e., $x$ denotes the
signed distance from the nearest point on $S,$ which then has
coordinate $y \in S.$ We let $\xi,\eta$ denote canonical dual
variables to $(x,y)$ in $T^*X.$

In the following proposition, we construct an extension operator
$\mathcal{E}$ that builds a microlocal solution to
$(P_h-z_h)=O(h^\infty)$  given data on $S$ (Cauchy data).

\begin{prop}\label{proposition:extension}

There exists an amplitude $a$ and a phase 
$$\phi=( (y-y')\cdot \eta \,+\, 
  x\sqrt{1-\abs{\eta}_{g(0,y)}^2}+O(x^2))
$$
having phase variable $\eta \in \RR^{n-1}$ such that the operator
$\mathcal{E}$ with kernel defined by 
$$
\ext (x,y,y') \equiv \big(\frac{h}{\sqrt{z_h}}\big)^{-(n-1)}\int a(x,y,y',\eta;h ) e^{i \phi\sqrt{z_h} /h} \, d\eta
$$
solves 
$$
\begin{aligned} (P_h-z_h) \ext f_h &=O(h^\infty),\\
 \ext f_h \restrictedto_{x=0} &= f_h+O(h^\infty),
\end{aligned}
$$
for any $f_h$ such that $\WF(f_h) \subset \{ \smallabs{\eta}< \frac{1}{2}\}.$
The amplitude $a$ enjoys an asymptotic expansion in nonnegative powers of
$h/\sqrt{z_h},$ with
coefficients that are smooth functions of $y',p.$
\end{prop}

\begin{proof}
We employ the Ansatz
$$
\bigg(\frac{\sqrt{z_h}}{h}\bigg)^{n-1}\int a(x,y,y',\eta;h/\sqrt{z_h}) e^{i \phi\sqrt{z_h}/h} f(y')\, d\eta \, dy',
$$
where $a$ is assumed to have an asymptotic expansion
$$
a\sim \sum_{j=0}^\infty a_j (x,y,y', \eta) \bigg( \frac{h}{\sqrt{z_h}} \bigg)^j.
$$
We find that the Cauchy data is reproduced (modulo $O(h^\infty)$) so
long as \begin{equation}\label{cauchydata} a\restrictedto_{x=0} \equiv 1\end{equation} and
$\phi\restrictedto_{x=0}=(y-y')\cdot \eta.$  On the other hand, applying
$P_h-z_h$ to this expression yields, first, the eikonal equation
$$
z_h\big( (\pa_x \phi)^2 + \abs{\nabla_y \phi}^2_g\big)-z_h=0,
$$
which as usual can locally be solved in the form
$$
\phi=( (y-y')\cdot \eta +
  x\sqrt{1-\abs{\eta}_{x,y}^2}+O(x^2)),
$$
by parametrizing the Lagrangian given by flowout along the Hamilton flow in
$(x,\xi,y,\eta)$ of the set
$$\{(x=0,\ \xi=\sqrt{1-\smallabs{\eta}^2_{g(0,y)}},\ y=y',\ \eta=\eta')\}.$$

Next, the leading transport equation reads
$$
2i h \sqrt{z_h} \nabla \phi \cdot \nabla a_0+ih\sqrt{z_h} \Lap (\phi) a_0=0.
$$
This can be solved by integrating from $x=0$ to give a smooth solution
with Cauchy data \eqref{cauchydata}.

The next transport equation picks out the term $(h/\sqrt{z_h}) a_1$ and now reads
$$
2i h\big(\frac{h}{\sqrt{z_h}}\big)  \sqrt{z_h} \nabla \phi \cdot \nabla a_1+ih\sqrt{z_h}\big(\frac{h}{\sqrt{z_h}}\big) \Lap (\phi)
a_1-h^2 \Lap(a_0)=0;
$$
this likewise has a smooth solution $a_1.$  Subsequent transport equations take the same form.

We may now Borel sum the results of solving the transport equations to
find that the operator $\ext$ with kernel
$$
\ext (x,y,y') \equiv \big(\frac{h}{\sqrt{z_h}}\big)^{-(n-1)}\int a(x,y,y',\eta;h ) e^{i \phi\sqrt{z_h} /h} \, d\eta
$$
solves $$\begin{aligned} (P_h-z_h) \ext f_h &=O(h^\infty),\\
 \ext f_h \restrictedto_{x=0} &= f_h+O(h^\infty).\qed
\end{aligned}
$$  
\noqed
\end{proof}

This extension operator gives us a way to parametrize (microlocally)
any solution to $(P_h-z_h) u_h =O(h^\infty).$

\begin{prop}\label{prop:u=ERu}
Let $u_h$ be a solution to $(P_h-z_h)u_h=O(h^\infty)$ such that $\WF_h
u_h \subset \{\xi>0 \}.$  Set $f_h=
u_h\restrictedto_{x=0}.$ Then microlocally near $(p_0, \dot{\gamma}_{ij}^k),$
\[
u_h-\ext f_h=O(h^\infty).
\] 
\end{prop}

\begin{proof}
The proof relies on the
microlocal energy estimates of \cite[Section 3.2]{Ch:11} (based, in
turn on \cite[Sections 23.1-23.1]{Hormander3} in the homogeneous
setting). According to these estimates, for any solution to $(P_h-z_h)
w_h =O(h^\infty)$ with $\WF' w_h \subset \{\xi>0\}$ we have the bound 
\begin{equation}\label{Cauchyestimate}
\norm{w_h \restrictedto_{x=x_1}} \lesssim h^{-C \abs{x_1}} \norm{w_h\restrictedto_{x=0}};
\end{equation}
the only real change needed from the treatment in
\cite{Ch:11} stems from the fact that operators $A_\pm$ used there are no longer
self-adjoint, with $A_\pm^*-A_\pm =O(\hlh),$ leading to the growth in
norms in the equation above: the LHS of equation (3.10) of
\cite{Ch:11} now has a factor of $h^{C y}$ inside the supremum
arising from the non-self-adjointness.

Equation \eqref{Cauchyestimate} suffices to 
show that our solution
$$
w_h=u_h-\ext f_h
$$ to 
$$\begin{aligned} (P_h-z_h) w_h &=O(h^\infty),\\
w_h \restrictedto_{x=0} &= O(h^\infty)
\end{aligned}
$$ 
must itself be $O(h^\infty),$ as desired.
\end{proof}


We can now use standard FIO methods first to change the phase function  
to the Riemannian distance along the geodesic $\gamma$ and then to extend $\ext$ in a microlocal neighbourdhood of
$\gamma$ even past conjugate points.

Changing the phase function means that, for $(x,y)$ strictly away from $S$ but still in a small neighborhood,
we may also write $u_h$ in the form

\begin{multline}\label{extension}
u_h(p) = \int \bigg( \frac{h}{\sqrt{z_h}}\bigg)^{-(n-1)/2}\tilde{a}
(p,y'; h/\sqrt{z_h}) e^{ i\sqrt{z_h}\dist(p,y')/h}f_h(y';h)\, dy'
\\ \,+O(h^\infty)
\end{multline}
where $\dist$ denotes the Riemannian distance.  Again the amplitude
has an expansion in integer powers of $h/\sqrt{z_h}.$

To see that this is possible, it suffices to observe that we
  may do a stationary phase expansion in the $y',\eta$ integrals
  (Proposition~\ref{proposition:statphase} below),
reducing
  the number of phase variables
  from $(n-1)$ to zero.  In this case, by construction the value
  of $d_{x,y}\phi$ where $d_\eta \phi=0$ is just the tangent to the
  geodesic flowout from $(x=0,\xi, y,\eta)$ where $\xi^2
  +\smallabs{\eta}^2=1,$ so that it agrees with $d_{x,y} \dist((x,y),(0,y')),$
hence the new
  oscillatory term becomes exactly $e^{i\dist \sqrt{z_h}/h}.$
  The amplitude has an expansion in powers of
  $h/\sqrt{z_h},$ with an overall factor of $(h/\sqrt{z_h})^{(n-1)/2},$
  since the $\eta$ integral was over $\RR^{n-1}.$


We now proceed to extend the
structure theorem for $\ext: f_h \to u_h$ globally along a given
geodesic, even past conjugate points.
\begin{prop}
Fix a geodesic $\gamma$ intersecting $S$ orthogonally at
$p_0\in S.$  Let
$p\in \gamma$ not be conjugate to $p_0.$ Subject to the
assumptions of Proposition~\ref{proposition:extension}, microlocally
near $(p, \dot\gamma)$ we have $u_h=\ext (f_h)$ with
\begin{equation}\label{extensionparametrix}
\ext f_h =\int \bigg( \frac{h}{\sqrt{z_h}}\bigg)^{-(n-1)/2}  a(p,y',\eta; z_h, h) e^{ i \dist(p,y')\sqrt{z_h}/h} f_h(y')\, d\eta \, dy'.
\end{equation}
The function $\dist$ should be interpreted here (potentially far
beyond the injectivity radius) as the smooth function given by
distance along the family of geodesics remaining microlocally close to
$\gamma,$ i.e.\ specified by the locally-defined smooth inverse of the
exponential map.
\end{prop}

\begin{proof}
Using the preceding construction, for any $x_0$ small enough we can
construct $\ext_{x_0}$ starting from the surface $x=x_0.$ Denoting by
$\res_{x_0}$ the restriction to the surface $x=x_0$ we then have, by
construction the following semigroup property :
\begin{equation}\label{eq:semigroup}
\ext_x \res_x \ext = \ext.
\end{equation}
The proof then follows by decomposing the geodesic into small enough
steps $[x_i,x_{i+1}]$ and applying stationary phase.
\end{proof}

\begin{defi}
  Since we will often be referring to symbols that have a half-step
  polyhomogeneous expansion in $h/\sqrt{z_h}$ in the product type case
  (i.e., with metric of the form \eqref{productmetric} near the boundary),
  but only have a leading order asymptotics modulo
  $O((h/\sqrt{z_h})^{1/2-0})$ times the leading order power in the
  general case, we will simply say that the function in question has
  \emph{adapted half-step asymptotics} to cover both cases.
\end{defi}
This distinction will not be of great importance for the results
presented here, but we will maintain it in the hope of future
applications in which the product type case may offer stronger results.

\subsection{Undergoing one diffraction}

Now we study the composition of the microlocalized extension operator
and the microlocalized wave propagator when one diffraction occurs. 

Let $\gamma$ be a geodesic orthogonal to $S$ at $p_0 \in S$ and terminating at cone
point $Y_j,$ with $p_0$ not conjugate to $Y_j$ (see \cite{BaWu:13} for
a definition of conjugacy in this context).  Away from the conjugate locus of $p_0$ let $\dist_S$ be
the distance function from $S$ measured along geodesics near
$\gamma.$  Let $\dist_j$ denote the distance function from the cone
point $Y_j,$ and let $\dist_{S,j}$ denote the restriction of $\dist_j$
to $S,$ near $p_0.$
\begin{prop}
Let $\WF_h' A$ and $\WF_h' B$ contain only diffractively related
points, and let $\sqrt{z_h} \in \Omega_\ep.$   There exists a symbol $c\in S_0$ with adapted half-step asymptotics
such that for \begin{equation}\label{sumofdistances}\dist_j(x) <t<\dist_{S,j}(p_0)+\dist_j(x),\end{equation} we have
$$A U(t) B \ext (x,y')=  c(x,y';h/\sqrt{z_h}) 
e^{i(\dist_{S,j}(y')+\dist_j(x)-t)\sqrt{z_h}/h}.$$
\end{prop}
\begin{proof}
We will employ stationary phase to
compose the two oscillatory integral
representations \eqref{FWpropagator} and \eqref{extension}.  In particular, we must evaluate an integral of the form
$$
\bigg( \frac{h}{\sqrt{z_h}}\bigg)^{-(n-1)/2}
\int\int_0^\infty a_D e^{i (\dist_j(x)+ \dist_j(x'')-t)\xi} a_S e^{i \dist_S
  (x'',y')\sqrt{z_h}/h}\, d\xi \, dx'',
$$
where $a_S$ has a complete asymptotic expansion in $\sqrt{z_h}/h$
while $a_D$ has adapted half-step asymptotics.
We would like to formally make the change of variables $\xi=\xi'\sqrt{z_h}/h$ and
then do stationary phase (as in Appendix~\ref{appendix:statphase}) in the small parameter $h/\sqrt{z_h}
\downarrow 0;$ justifying this deformation into the complex in fact
proceeds as follows.  To begin, we let $\eta=h\xi$ so that we are
trying to evaluate
$$
h^{-1}\bigg( \frac{h}{\sqrt{z_h}}\bigg)^{-(n-1)/2}
\int\int_0^\infty a_D e^{i (\dist_j(x)+ \dist_j(x'')-t)\eta/h} a_S e^{i \dist_S
  (x'',y')\sqrt{z_h}/h}\, d\eta \, dx''.
$$
By the usual method of nonstationary phase, we then find that the
integral is unchanged modulo $O(h^\infty)$ if we insert a compactly
supported cutoff $\chi(\eta),$ equal to $1$ for $\abs{\eta}<2.$
Finally, we replace $\chi(\eta) a_D(\dots,\eta) $ with an almost-analytic
extension in $\eta,$ and set $\xi'=\eta/\sqrt{z_h},$ justifying the resulting
contour deformation exactly as in the proof of
Lemma~\ref{lemma:quadphase} in the appendix.  This finally yields
$$
\bigg( \frac{h}{\sqrt{z_h}}\bigg)^{-(n+1)/2}
\int\int_0^\infty a_D e^{i (\dist_j(x)+ \dist_j(x'')-t)\xi'\sqrt{z_h}/h} a_S e^{i \dist_S
  (x'',y')\sqrt{z_h}/h}\, d\xi' \, dx''.
$$
Finally we apply Proposition~\ref{proposition:statphase} to justify a
formal stationary phase expansion of this expression.
Stationary points are
where
\begin{align}
  \label{eq:1}
  \dist_j(x)+ \dist_j(x'')-t&=0,\\
 \nabla_{x''} \dist_j(x'') +\xi' \nabla_{x''} \dist_S(x'',y')&=0.
\end{align}
The latter equation implies that $\xi'=1$ and that $x''$ must lie on
the unique geodesic near $\gamma$ connecting $y'$ to $Y_j,$ hence at the stationary
point, $\dist_j(x'')+
\dist_S(x'',y')= \dist_{S,j}(y').$  Thus stationary phase yields a result
of the form
$$
\text{symbol} \cdot e^{i(\dist_{S,j}(y') + \dist_j(x)-t)\sqrt{z_h}/h}.
$$
Since the stationary phase is in the $n+1$ variables
$(x'',\xi)$ we gain an overall factor $(h/\sqrt{z_h})^{(n+1)/2},$
reducing the overall power of $h/\sqrt{z_h}$ to zero.
\end{proof}

We can extend this formula for larger times by precomposing and
postcomposing with $U(t)$ (see \cite[Section 5]{FoWu:17}). Using the
semigroup property for $U(t)$, and stationary phase to compute the
compositions, we see that this formula continues to
hold, correctly interpreted for $x$ far from $Y_j$ as well, as long as
$x$ is not conjugate to $Y_j$ and 
we microlocalize near the geodesic connecting $Y_j$ to $x.$
\begin{prop}
  Let $A$ be microsupported sufficiently close to a geodesic $\gamma$ coming
  from $Y_j,$ in a small neighborhood of a point not conjugate to
  $Y_j$ along $\gamma.$  There exists a symbol $c\in S_0$ with adapted half-step
  asymptotics such that for
$$\dist_j(x) <t<\dist_{S,j}(p_0)+\dist_j(x),$$
$$A U(t) \ext (x,y')= c(x,y';h/\sqrt{z_h}) 
e^{i(\dist_{S,j}(y')+\dist_j(x)-t)\sqrt{z_h}/h},$$
where the distance is interpreted as distance along geodesics
microlocally close to $\gamma.$
\end{prop}

Finally, we examine what happens when we again restrict to a
hypersurface.  Let the geometric setup be as above, with $S'$ a
hypersurface orthogonal
to a geodesic from $Y_j.$  Let $\res$ denote the operation
of
restriction to $S',$ and let $\dist_{j,S'}$ denote the distance
function from the cone point $Y_j$ to $S'.$
\begin{prop}\label{proposition:structure}
There exists a symbol $c\in S_0$ such that for 
$$\dist_j(y) <t<\dist_{j}(y')+\dist_j(y),$$
$$\res U(t) \ext (y,y')=  c(y,y';h/\sqrt{z_h})
e^{i(\dist_{S,j}(y')+\dist_{j,S'}(y)-t)\sqrt{z_h}/h}.$$
\end{prop}

\section{Monodromy data}

In this section we examine relations that hold among the restrictions
of a resonant state to cross sections of geodesics in $\Gamma.$

Consider the directed multigraph whose vertices are the cone points
and whose edges are the oriented geodesics connecting pairs of cone
points.  Let $E$ be the edge set of this graph and for $e,f \in E$
write $e\to f$ if $e$ and $f$ are adjacent in the sense of digraphs,
i.e.\ if $e$ terminates at some vertex $Y_ i$ while $f$ emanates from
$Y_i.$ For any edge $e$ let $\bar e$ denote the edge corresponding to
the same geodesic but with opposite orientation.

 To each
directed edge $e$ we fix a patch of oriented hypersurface $S_e$
intersecting it orthogonally at a point not conjugate to either of the cone
points at which $e$ originates and terminates.  There is no
particular need to require that $S_e$ and $S_{\bar e}$ be identical as unoriented
surfaces (but of course that is one option).  We arrange for the sake
of simplicity that each $S_e$ intersects only the edges $e, \bar e.$

For each $e\in E$ let $d_e^\pm$ denote the functions on $S_e$ given by the
distances to the cone points at the end point ($+$) and starting point
($-$) of $e$.  Let $\len_e$ denote the length of the edge $e$ (hence
of course $\len_e=\len_{\bar e}$).

For each edge $e\in E$, let $\ext_e$ resp.\ $\res_e$ respectively
denote the parametrix for the extension operator
\eqref{extensionparametrix} and the restriction operators from/to the
oriented surface $S_e$ orthogonal to this edge (as discussed in
\S\ref{section:restriction}).  Given $f,e \in E$ with $f \to e,$ with
incidence at the cone point $Y_j$ let $t_{fe}$ denote a number
exceeding $\dist(S_f, Y_j)+\dist(Y_j, S_e)$ by a small fixed quantity
$\epsilon_1>0.$

\begin{lemma}
Consider a sequence of solutions to
$$
(P_\theta-z_h) u_h=0, \quad h \downarrow 0.
$$
For each $e \in E,$
\begin{equation}\label{extensionrestriction}
\res_e u_h -  \sum_{f \to e} e^{it_{fe} \sqrt{z_h}/h}  \res_e U(t_{fe}) \ext_f \res_f u_h=O(h^\infty).
\end{equation}
\end{lemma}
\begin{proof}
We first recall that, for any $t,$ 
$$
\res_e u_h =\res_e e^{it\sqrt{z_h}/h} U(t) u_h+O(h^\infty);
$$
choose $t=d+\eps$ where $d$ is the distance from $S_e$ to the
cone point $Y_j$ from which the edge $e$ emanate and $\eps$ is small
enough so that that $t< t_{fe}-\eps$ for all $f\to e.$  By propagation of
singularities (Proposition~\ref{proposition:propofsings}) $\res_e u_h$
is determined, modulo $O(h^\infty),$ by $u_h$ on the sphere of radius
$\eps$ centered at $Y_j.$ Denote by $\Vcal$ the $\frac{\eps}{2}$
neighborhood of this sphere. Since $\WF_h u_h$ is a subset of the geodesics
represented by the edges in $E,$ we may let $A_f$ denote
microlocalizers along all edges $f \to e,$ supported in $\Vcal$ with
$\WF_h'(I-A_f)$ disjoint from all 
points near $f$ that are diffractively related to
points in $S_e$ in time $t.$
By the semiclassical Egorov theorem \cite[Theorem
11.1]{Zw:12} for $e^{-is\sqrt{\Lap}}=e^{-is\sqrt{h^2\Lap}/h}$ (with
$s=t_{fe}-t$), we may write
$$
A_f U(t_{fe}-t)=U(t_{fe}-t) A_f'
$$
where
$A_f'$ is microsupported near the intersection of $f$ with $S_f$
with $\WF_h'(I-A_f')$ disjoint from a smaller neighborhood of this intersection.

By Proposition~\ref{proposition:propofsings}, then,
\begin{equation}\label{recurrence}
\res_e u_h =\sum_{f\to e} \res_e e^{it\sqrt{z_h}/h} U(t) A_f u_h+O(h^\infty).
\end{equation}
and by \eqref{reproducing} and a further application of Proposition~\ref{proposition:propofsings},
\begin{align*}
\res_e e^{it\sqrt{z_h}/h} U(t) A_f u_h &=\res_e e^{it_{f e}\sqrt{z_h}/h} U(t) A_f U(t_{f e}-t) u_h+O(h^\infty)\\
&= \res_e e^{it_{f e}\sqrt{z_h}/h} U(t_{f e}) A_f' u_h+O(h^\infty) \\
&=\res_e e^{it_{f e}\sqrt{z_h}/h} U(t_{f e}) \ext_f \res_f u_h+O(h^\infty).\qed
\end{align*}
\noqed
\end{proof}

Thus we conclude that the functions  $\res_e u_h$ satisfy a set of
relations which we can now employ to deduce constraints on $\Im z_h.$
By Proposition~\ref{proposition:structure} we find the
following:
\begin{prop}
There exist symbols $c_{fe}$ of order zero with adapted half-step asymptotics such
that for each $e,$
\begin{equation}\label{ftoe}
\res_e u_h = \sum_{f\to e} A_{fe} \res_f u_h+O(h^\infty)
\end{equation}
where $A_{fe}$ has Schwartz kernel
\begin{equation}\label{A}
A_{fe}(y,y') =  c_{fe}(y,y';h/\sqrt{z_h})  e^{i (d_e^-(y)+d_f^+(y')) \sqrt{z_h} /h}.
\end{equation}
\end{prop}
\begin{proof}
We insert the representation of the wave propagator from
Proposition~\ref{proposition:structure} into \eqref{extensionrestriction}.
\end{proof}

Thus we have a matrix equation with operator-valued entries for the
restriction to the hypersurfaces.

\begin{prop}\label{proposition:formofrestriction} 
Let $M$ be the constant used in defining the scaling region and $L_0$
the longest geodesic between two cone points. For each $e,$ there exists a smooth
  amplitude $s_e(y;h)\in h^{-1-ML_0} S(1)$ such that
$$
\res_e u_h = e^{i d_e^-(y)\sqrt{z_h}/h} s_e(y;h)+O(h^\infty).
$$
\end{prop}

\begin{rema}
No claim is made here about polyhomogeneity of $s_e$ in $h.$
\end{rema}

\begin{proof}
We want to prove that $y\mapsto s_e(y;h)$ is smooth and that any seminorm
$\| \partial^\beta s(\cdot \,;\,h)\|_\infty$ is $O(h^{-1-ML_0}).$ The fact
that $s$ is smooth follows from the fact that $u_h$ is smooth by
ellipticity. Using the eigenvalue equation we also have 
$\| u_h \|_{H^k_\loc} = O (h^{-k}),$ so that, roughly speaking, we
lose one power of $h$ by differentiating. The content of the proposition is
thus that this loss actually does not occur.

We apply the operator given in \eqref{A} to $u_h,$ noting that
we can pull the factor $ e^{i d_e^-(y)\sqrt{z_h} /h}$ out
of the integral.  What remains is to show that the remaining factors of the form
\begin{equation}\label{smoothamplitudes}
\tilde{s}_{fe} (y;h) \,=\, \int  c_{fe}(y,y';h/\sqrt{z_h})  e^{id_f^+(y') \sqrt{z_h}
  /h} (\res_f u_h)(y') \, dy'
\end{equation}
are in fact smooth amplitudes.


Using Sobolev embedding, and the fact that $u_h$ is a solution to the
eigenvalue equation, we have 
\[
\| \res_f u_h\|_{L^2(y)} \,\leq\, C h^{-1}. 
\]
Replacing in \eqref{smoothamplitudes} and using the Cauchy-Schwarz
inequality  we obtain  
\[
| \tilde{s}_{fe}(y,h)| \,\leq\, Ch^{-1} h^{-ML_0} \left( \int |c_{fe}(y,y';
\frac{h}{\sqrt{z_h}})|^2 dy'\right)^{\frac{1}{2}}.
\]
The integral is uniformly bounded (in $h$) owing to the fact that
$c_{fe}$ is a symbol.
Moreover, $\tilde{s}_{fe}$ enjoys
iterated regularity under differentiation in $y,$ as $y$ derivatives
of \eqref{smoothamplitudes} only hit the factor $c_{fe},$ hence all
$y$-derivatives are $O( h^{-1-ML_0}).$
\end{proof}

It is convenient to rewrite \eqref{ftoe} in terms of the amplitudes
$s_e$ :
\begin{equation}\label{eq:s=Ms}
s_e\,=\, \sum_{f\rightarrow e} e^{i\ell_f\frac{\sqrt{z_h}}{h}} \Mcal_{fe} s_f, 
\end{equation}
where $\Mcal_{fe}$ has Schwartz kernel 
\begin{gather}\label{def:Mef}
M_{fe}(y,y') =  c_{fe}(y,y';\frac{h}{\sqrt{z_h}})  
e^{i \frac{\sqrt{z_h}}{h}\phi_f(y,y')}~~~\\
\nonumber \mbox{with}~~\phi_f(y,y') = d_f^+(y')+d_f^-(y')-\ell_f.
\end{gather}
Observe that the phase $\phi_f$ has a unique non-degenerate critical point at
$y'=0$ and satisfies $\phi_f(0)=0.$ 

Equation \eqref{eq:s=Ms} can be rewritten as  
\[
\left( \Id-\Mcal \right) \mathbf{s} \,=\,O(h^\infty),
\]
which is typical of a monodromy operator in such settings.   In our setting, we can pass to a discrete set of restriction data, given by the jets
of the restriction to $S_e$ at $e.$  (Compare, e.g., Proposition 4.3 of
\cite{Burq:Coin} and see also \cites{NSZ:11,Nonn:12}). This will reduce the monodromy
equation to a finite dimensional system.  

For each $\alpha \in \NN^{n-1},$ we thus let
$$
s_e^\alpha=\pa_y^\alpha s_e,
$$
so that
$$
s_e(y) = \sum_{\smallabs{\alpha}<N}  s_e^\alpha \frac{y^\alpha}{\alpha
  !} +R_{e,N},\quad R_{e,N}=O(\smallabs{y}^N),
$$

Applying the stationary phase computation in
Appendix~\ref{appendix:statphase} gives the asymptotic expansion 
\begin{equation}\label{eq:Myalpha}
\Mcal_{fe} ({y'}^\alpha)\,\sim\, \left(
  \frac{h}{\sqrt{z_h}}\right)^{\frac{n-1}{2}}
\sum_{k\geq \frac{|\alpha|}{2}} m_{ef\alpha k} (y; \frac{h}{\sqrt{z_h}}) \left( \frac{h}{\sqrt{z_h}}\right)^k,
\end{equation}
where $m_{ef\alpha k}$ is a symbol.
 
\begin{rema}
If $|\alpha|$ is odd, a closer inspection shows that this
contribution is $O(h^\infty).$
\end{rema}

\begin{prop}\label{proposition:relation}
\mbox{}\begin{enumerate}\item
For each $f\to e,$ $\alpha,$ $\beta$ there exist
$C(\alpha,\beta, e,f;h),$ bounded in $h\downarrow 0,$ such that
\begin{equation}\label{relation}\begin{aligned}
s_e^\alpha=\sum_{f\to e}\sum_{j<N} \sum_{\smallabs{\beta}\leq 2j}
C(\alpha,\beta,j,e,f;\frac{h}{\sqrt{z_h}})
&\big( \hsh \big)^{(n-1)/2+j} e^{i \len_f \zsh}
s_f^\beta\\ &+O\bigg(\big( \hsh \big)^{(n-1)/2+N} e^{i \len_f \zsh}\bigg).\end{aligned}
\end{equation}
The coefficients
$C(\alpha,\beta,j,e,f; \hsh)$ enjoy adapted half-step
asymptotics, and in particular
$$
C(\alpha, 0,0, e, f; \hsh)\equiv C(\alpha, e, f)+O((\hsh)^{1/2-0}),
$$
with $C(\alpha, e, f)$ independent of $\hsh.$
\item 
For all $m_0>0$ there exist $m_1>0$ and $N_1 \in \NN$ such that
whenever $s_e^\alpha=O(h^{m_1})$ for all $e\in E,$ $\smallabs{\alpha}
\leq N_1,$ we have
$$
s_e(y)=O(h^{m_0}) \text{ for all } e \in E.
$$
\end{enumerate}
\end{prop}

\begin{proof}
We truncate the asymptotic expansion \eqref{eq:Myalpha} at $k=N.$ 
We plug it into \eqref{eq:s=Ms} and then extract
the coefficient $s_e^\alpha$ from this expansion. This gives
\eqref{relation}. The assertions on the coefficients follow by inspection. 
For the last assertion, we first observe that 
\[
|e^{i \len_f \zsh}|\,=\,O(h^{-ML_0}),
\]
where $M$ is the constant used to define the scaling region and $L_0$
the longest geodesic between cone points. We then choose $N$ so that 
$ \frac{n-1}{2}+N-ML_0 \geq m_0$. We then set $N_1=2N,$ and
$m_1=m_0-\frac{n-1}{2}$, so that in \eqref{relation} all the terms are 
$O(h^{m_0}).$  
\end{proof}

\section{Proof of Theorem~\ref{theorem:gap}}\label{section:proof}

We now prove Theorem~\ref{theorem:gap}.

Fix any $\ep>0.$  We assume throughout that
$$\sqrt{z_h} \in \Omega_\ep\equiv \{
  (-\Lambda +\epsilon)\hlh <\Im\sqrt{ z_h}<0,\ \Re \sqrt{z_h} \in [1-\ep,1+\ep] \}.$$
Our aim is to show that if $z_h$ is such a
  sequence of resonances, with $\Re \sqrt{z_h} \to E$ and and $\Im
  \sqrt{z_h} \sim -\nu \hlh$ then we must have $\nu= (n-1)/2L_0,$
  while $\Re \sqrt{z_h}$ satisfy a quantization condition.  As discussed
  above, we fix $M,$
  the parameter in our complex scaling, with $M \gg \Lambda$ so that
  square roots of eigenvalues of $P_\theta$ in $\Omega_\ep$ agree with resonances of
  $P$ in that set.  By
  Proposition~\ref{proposition:variablepropagation} and
  Proposition~\ref{proposition:propofsings}, there exists some $m_0>0$
  such that if $\WF^{m_0} u_h \cap \Gamma=\emptyset,$ then $\WF^\ep
  u_h=\emptyset,$ hence $u_h$ could not possibly be an
  $L^2$-normalized resonant state. 
It
  thus suffices to show that if $\sqrt{z_h} \in \Omega_\ep$ does not satisfy the
  quantization condition or the condition on the imaginary part, we
  must have $\WF_h^{m_0} u_h\cap \Gamma=\emptyset$ for this fixed,
  potentially large, $m_0>0.$  Moreover, again by
  Proposition~\ref{proposition:variablepropagation}, it suffices to
  show absence of $\WF^{m_1} u_h$ at the intersection of
  $\Gamma$ with the edges $e \in E,$ for some potentially larger
  $m_1>0.$  To show this, in turn, we see by
  the second part of Proposition~\ref{proposition:relation} that it suffices to show that
if the desired conditions on $z_h$ are not met then each $s_e^\alpha$ is $O(h^{m_2})$ for all $\alpha$ with
  $\abs{\alpha}\leq N_2$ for some large (but geometrically
  determined) $N_2.$

 Thus, we suppose that either
  the quantization condition or the imaginary part condition is
  \emph{violated} and we aim to show that consequently
  $s_e^\alpha=O(h^{m_2})$ for this finite list of values of $\alpha.$

For any $N>N_2,$ we let $A_N$ denote the weighted directed edge adjacency
matrix for a multigraph with multiple edges $(e,\alpha)$ for each
edge $e\in E$ as above, but now with $\alpha \in \NN^{n-1}$ a multi-index, with
$\smallabs{\alpha}<N$ and
with $(e,\alpha) \to (f,\beta)$ an adjacency iff $e\to f$ in our
original multigraph of directed geodesics with edge set $E.$
The $(e,\alpha, f,\beta)$ entry of
$A_N$ may thus be nonzero only if $f \to e,$ and in that case is defined
to be:
$$
(A_N)_{e,\alpha,f,\beta} \equiv \sum_{j \in \NN: \smallabs{\beta}\leq 2j<2N} C(\alpha,\beta,j, e,f;\hsh)
\big( \hsh \big)^{(n-1)/2+j} e^{i \len_f \zsh},
$$
with $C(\alpha,\beta, j,e,f;\hsh)$ given by Proposition~\ref{proposition:relation}.

\begin{lemma}\label{lemma:matrixrelation}
If $\sqrt{z_h} \in \Omega_\ep,$ 
and
$$
(P_\theta-z_h)u_h=0
$$
then viewed as an $h$-dependent vector $\bs=\{s_e^\alpha\},$ the Cauchy data $s_e^\alpha$ of $u_h$ on the hypersurfaces $S_e$
satisfies
\begin{equation}\label{relation2}
\bs =A_N \cdot \bs+\br,\quad \br =O(h^{(n-1+N)/2-L_0(\Lambda -\epsilon)}).
\end{equation}
\end{lemma}
\begin{proof}
This is just a truncation of the result of
Proposition~\ref{proposition:relation}, where we also use the fact
that for $\sqrt{z_h} \in \Omega_\ep,$
$z_h, z_h^{-1}=O(1)$ while $\Im \sqrt{z_h}>(-\Lambda+\epsilon)\hlh,$
hence $$\abs{e^{i \len_f \sqrt{z_h}/h}}=O(h^{-\len_f(\Lambda-\ep)})
=O(h^{-L_0 (\Lambda-\ep)}),$$ while the $s_e^\alpha$ are
all bounded as $h \downarrow 0.$
\end{proof}

Since the entries in $A_N$ are all seen to be
$O(h^{(n-1)/2-L_0(\Lambda-\ep)})$ by the same reasoning as in the
proof of Lemma~\ref{lemma:matrixrelation}, we find applying $A_N$ to
both sides of our relation that 
\begin{equation}\label{squaredmatrix}
\begin{aligned}
A_N^2 \bs &=
A_N \bs -A_N \br\\
&=\bs-\br -A_N\br\\
&=\bs +
O(h^{(n-1+N)/2-L_0(\Lambda-\ep)})+O(h^{(n-1)+N/2-2L_0(\Lambda-\ep)})\\
&=\bs +O(h^{(n-1+N)/2-L_0(\Lambda-\ep)}),
\end{aligned}
\end{equation}
since $\Lambda \leq (n-1)/2L_0.$

Now we examine the entries in $A_N^2.$ This matrix has
diagonal entries, corresponding to two-cycles in our digraph.  The
largest such entries are the two with zero multi-index that correspond to traversing $f,\bar
f$ or $\bar f, f,$ for $f$ a maximal edge, i.e., $\len_f=L_0.$ 
{We assume from now on that there are $J$ such maximal edges.}By
assumption, no two geodesics with length $L_0$ are incident on the
same cone point. It is thus impossible to have an off-diagonal entry in
$A^2_N$ with a maximal contribution and the only way to obtain
these largest entries 
is to traverse one maximal oriented geodesic {$f_j$} and then return whence we came
on {$\bar f_j.$} 

The remaining entries are all bounded by either
$$
O( h^{(n-1)-(2L')(\Lambda-\ep)})
$$
in the case in which at least one edge traversed is not maximal, or
else by
$$
O( h^{(n-1)+1-2L_0 (\Lambda-\ep)})
$$
in the case that one multi-index $\beta$ (and hence one value of $j$) is nonzero.  In either case, we find
that the entry is bounded by $O(h^{\ep'})$ with 
\begin{equation}\label{epsilonprime}
\ep' \equiv \min\big \{(n-1)-2L'(\Lambda-\ep), (n-1)+1/2-2L_0 (\Lambda-\ep)\big\}.
\end{equation}
Note that this is a positive number, by definition of $\Lambda;$ the
$1/2$ term in the second expression is not optimal in the argument
above but will be necessary below.
Thus, if we write $A_N$ in block form with the edges $(f,0)$ and
$(\bar f, 0)$ with $f$ maximal listed first, we have
$$
A_N^2=Q_N^0+O(h^{\ep'})
$$
with
{$$
Q_N^0 = \left( \begin{array}{cccccc} \tilde{B}_1 &  &  & & \\
 & \ddots &  &  & \Scale[1.8]{0} & \\
 &  & \tilde{B}_J & & &\\
& & & 0 & &\\
& \Scale[1.8]{0} & &  & \ddots &\\
& & & & & 0
\end{array}\right)
$$}
where the $2 \times 2$ block {$\tilde{B}_j$} is given by
{$$
\tilde{B}_j= (A_N)_{(\bar f_j,0, f_j, 0)} (A_N)_{(f_j,0,\bar f_j,  0)} \Id_{2\times 2}
$$}
and where the number of these blocks equals the number of maximal
geodesics.  Replacing these matrix entries with their leading order
approximation, we now get the improvement
$$
A_N^2=Q_N+O(h^{\ep'})
$$
with
{$$
Q_N = \left( \begin{array}{cccccc} B_1 &  &  & &\\
 & \ddots &  &  &\Scale[1.8]{0} & \\
 &  & B_J & & &\\
& & & 0 & &\\
&\Scale[1.8]{0} & &  & \ddots &\\
& & & & & 0
\end{array}\right)
$$}
and where, {for each $j$} \footnote{{In the sequel,
    we will drop the index $j$ for readibility.}}
$$
B_j= C(0,\bar f_j, f_j) C(0, f_j, \bar f_j) h^{(n-1)} e^{2 i L_0
  \sqrt{z_h}/h}\Id_{2\times 2}.
$$

Here we have used the fact that approximating $(A_N)_{(f,0,\bar f,
  0)} $ by ${C(0, f, \bar f) h^{(n-1)/2} e^{ i L_0
  \sqrt{z_h}/h}}$ incurs an error of $O(h^{(n-1)/2+1/2-0} e^{ i L_0
  \sqrt{z_h}/h}).$  Hence for any $N$
sufficiently large, for  for $\sqrt{z_h} \in \Omega_\ep$ we have
\begin{align*}
(A_N)_{(\bar f,0, f, 0)} (A_N)_{(f,0,\bar f,  0)} - &C(0,\bar f, f) C(0, f, \bar f) h^{(n-1)} e^{2 i L_0
  \sqrt{z_h}/h}\\&=O(h^{(n-1)+1/2-0}e^{i 2L_0   \sqrt{z_h}/h})\\&=O(h^{\ep'}),
\end{align*}
with $\ep'$ given by \eqref{epsilonprime} above; this argument is where the
$1/2$ gain in the second term in the minimum taken in
\eqref{epsilonprime} is now optimal.

Now choose $N$ big enough that
$$
\frac{(n-1+N)}{2} -L_0 \Lambda \gg m_2+\ep',
$$
where we recall that $m_2$ is the decay rate required to show that $u_h$
could not be an $L^2$-normalized eigenfunction of $P_\theta.$
Then by \eqref{squaredmatrix} we have the simple equation
$$
(\Id-Q_N+O(h^{\ep'})) \bs = O(h^{m_2+\ep'}).
$$
Clearly, \emph{if} $\Id-Q_N$ is invertible with $(\Id-Q_N)^{-1}=O(h^{-\ep'/2})$
we can then invert to obtain
$$
\bs = O(h^{m_2+\ep'/2}),
$$
which is the desired estimate: if this holds, then $u_h$ could not
have been a normalized resonant state after all.

Thus in order for a resonance to {exist}, we must have, by contrast, a lower bound on the inverse of
$(\Id-Q_N)^{-1}.$ In particular, {for some maximal edge $f$, we must have}
$$
\abs{C(0,\bar f, f) C(0,f,\bar f) h^{(n-1)} e^{2 i L_0
  \sqrt{z_h}/h}-1}^{-1} \geq C h^{-\ep'/2},
$$
i.e.,
$$
C(0,\bar f, f) C(0,f,\bar f) h^{(n-1)} e^{2 i L_0
  \sqrt{z_h}/h} =1+O( h^{\ep'/2}).
$$
Taking $\log,$ this yields 
$$
2i L_0 \frac{\sqrt{z_h}}{h}+(n-1) \log h + \log C(0,\bar f, f) C(0,f,\bar f)\in 2\pi i \ZZ+ O(h^{\ep'/2}).
$$
The equality of imaginary parts yields for the constant $$C_{\Re}=-\Im
\log C(0,\bar f, f) C(0,f,\bar f)/2L_0,$$
$$
\Re \sqrt{z_h}  \in  h \big( C_{\Re} + \frac{\pi}{L_0} \ZZ \big) + O(h^{1+\ep'/2})
$$
while taking real parts gives for the constant $$C_{\Im}=\Re \log C(0,\bar f, f) C(0,f,\bar f)/2L_0$$
$$
\Im \sqrt{z_h}  =-\frac{(n-1)}{2L_0} \hlh + C_{\Im}h + O(h^{1+\ep'/2}).
$$
Recalling that semiclassical rescaling gave
$$
\Re\sqrt{z_h} =1,\ h=(\Re \lambda)^{-1},\ \frac{\Im\sqrt{z_h}}{h} = \Im \lambda,
$$
this yields the statements of the theorem.

\begin{rema}
As shown by Galkowski \cite{Ga:15}, if one of the diffraction coefficients
$C(0,\bar{f},f)$ does not vanish then the width of the resonance free 
logarithmic region depends on $L_0.$ Our argument shows when all these
coefficients vanish, we indeed obtain a larger resonance-free 
logarithmic region, so that this condition is sharp.  
\end{rema}

\appendix
\section{(Micro)-locality of $\sqrt{\Lap}$}\label{appendix:micro}
The aim of this appendix is to prove the two following facts for the
Laplace operator $\Delta$ on a manifold with conical singularities.

\begin{prop}\label{prop:appA}
Let $X$ be a manifold with conical singularities and $\Delta$ its 
self-adjoint Laplace operator (Friedrichs extension). Then 
\begin{enumerate}
\item \label{smoothing} For any open sets, $U,V\in X$ such that $U\cap V =\emptyset$, 
and $V\subset X^\circ,$ for any $N$, $\Delta^N\sqrt{\Delta}$ is
continuous from $L^2(V)$ into $L^2(U).$
\item \label{rootispseudo} For any open set $U$ such that
  $\overline{U}\subset X^\circ,$ 
$\sqrt{\Delta}$ seen as an operator from $H^1(U)$ into $L^2(U)$ 
is a (first order) pseudodifferential operator.    
\end{enumerate}
\end{prop}

Both results will follow from studying the heat kernel and using the
transform :
\[
\sqrt{\Delta} \,=\, \frac{\Delta}{\Gamma(\frac{1}{2})} \int_0^\infty e^{-t\Delta}
t^{-\frac{1}{2}}\, dt.
\]

First, we take $\rho \in C^{\infty}_0([0,+\infty))$
that is identically $1$ on $[0,2t_0]$ for some $t_0>0$ and write $\psi =1-\rho.$ 
Since 
\[
\int_0^\infty e^{-t\Delta} \psi(t) t^{-\frac{1}{2}}\, dt\,=\,
e^{-t_0\Delta} \int_0^{\infty} e^{-(t-t_0)\Delta}\psi(t) t^{-\frac{1}{2}}
  \, dt,
\]
we see that the operator that is defined by the former integral 
is smoothing. 

Hence, both claims will follow from the same claims where $\sqrt{\Delta}$
is replaced by 
\begin{equation*}
\Delta \int_0^\infty \rho(t) e^{-t\Delta} t^{-\frac{1}{2}}\, dt. 
\end{equation*}

Multiplying by $\Delta$ does not modify the
statements so that it suffices to study the operator-valued integral :

\begin{equation}\label{Irho}
\int_0^\infty e^{-t\Delta} \rho(t) t^{-\frac{1}{2}}\, dt
\end{equation} 
\hfill \\

\begin{proof}[Proof of \eqref{smoothing}] 
For any $a\in L^2(U)$, we define the distribution $T_a$ on $\RR\times
V$ by 
\[
\left( T_a , \phi(t) b(y) \right )_{\mathcal{D}'\times \mathcal{D}} \,=\, \int_0^\infty \langle a,
e^{-t\Delta}b\rangle_{L^2} \phi(t) dt.
\] 

Since $\lim_{t\rightarrow 0^+}\langle a,
e^{-t\Delta}b\rangle_{L^2} =0,$ a simple calculation shows that, in the distributional sense 
\[
(\partial_t + \Delta_y)T_a =0,~~\mbox{in}~~\mathcal{D}'(\RR\times V).
\]
By hypoellipticity in $\RR\times V$, it follows that $T_a$ is smooth.

Since $T_a$ vanishes identically for $t<0$
\[
\forall (a,b)\in L^2(U)\times L^2(V), ~t\mapsto \langle
e^{-t\Delta}a,b \rangle 
\] 
is smooth on $[0,\infty)$ and vanishes to infinite order at $0.$

In particular, for any $N$ and $k,$ the $N$-th derivative of the
latter function vanishes to order $k$ at $0.$ So the quantity 
\[
t^{-k} \langle \Delta^N e^{-t\Delta} a, b\rangle 
\]
is bounded on $(0,1].$

By the principle of uniform boundedness this implies that 
\[ \| \Delta^N e^{-t\Delta}\|_{L^2(V)\rightarrow L^2(U)} \,=\,O(t^k).\]

Plugging this bound into the integral \eqref{Irho} yields the result.
\end{proof}

\begin{proof}[Proof of \eqref{rootispseudo}]
We begin by choosing $\tilde{U}$ that is compactly embedded into $X^\circ$
and such that $U\subset \tilde{U}.$ We denote by $e$ the heat kernel
on $X$ and by $\tilde{e}$ the heat kernel on a complete smooth
Riemannian manifold $\tilde{X}$ in which $\tilde{U}$ is embedded.

We denote by $r$ the distribution on $\RR\times U\times U$ that is
defined by 
\[
\left( r, \phi\right)\,=\, \int_0^\infty \int_{U\times U}
(e(t,x,y)-\tilde{e}(t,x,y))\phi(t,x,y) dxdy~dt.
\] 
Observing that, for any $\phi$ 
\[
\int_{U\times U} (e(t,x,y)-\tilde{e}(t,x,y))\phi(t,x,y) dxdy
\underset{t\rightarrow 0}{\longrightarrow} 0,
\]
we obtain that, in $\mathcal{D}'(\RR\times U\times U)$ 
\[
(2\partial_t \,+\,\Delta_x\,+\,\Delta_y)r=0.
\]
So by hypoellipticity, $r$ is smooth on $\RR\times U\times U.$ 

Consequently, in \eqref{Irho}, if we replace $\Delta$ by the Laplace
operator on $\tilde{X}$, we make an error whose kernel is 
\[
\int_0^\infty \rho(t) r(t,x,y) t^{-\frac{1}{2}} dt.
\] 
Since $r$ is smooth and vanishes to infinite order at $t=0$ 
this integral is a smooth function of $x$ and $y.$

It follows that $\sqrt{\Delta}$ in $U$ coincides with
$\sqrt{\tilde{\Delta}}$ up to a smoothing operator.
\end{proof}

\section{Stationary phase}\label{appendix:statphase}

In this appendix we discuss the method of stationary phase when the
large parameter is allowed to be complex, with imaginary part
comparable to the logarithm of the real part.  We will parallel the
treatment and notation in \cite[Section 3.5]{Zw:12}.  The outcome will
be that we may treat the factor of $w_h$ below as a formal parameter, but we
have been unable to find a justification for these manipulations in
the published literature.

As before we write
$$\Omega_\ep\equiv \{w_h\colon
  (-\Lambda +\epsilon)\hlh <\Im w_h<0,\ \Re w_h \in [1-\ep,1+\ep] \}.
$$
For $a \in \CcI(\RR^n),$ $\varphi \in \CI(\RR^n)$ real valued and $w_h
\in \Omega_\ep,$ we define
$$
I_h(a,\varphi; w_h)\equiv \int_{\RR^n} e^{i\varphi w_h/h} \, a\, dx.
$$
Note that the exponential term may be polynomially growing in $h$
owing to the presence of the factor $w_h \in \CC.$  We
will use throughout the fact that $h/w_h=O(h)$ for $w_h \in \Omega_\ep.$

\begin{lemma}
Let $w_h \in \Omega_\ep.$ If $d \varphi \neq 0$ on $\supp a,$ then $I_h(a,\varphi; w_h)=O(h^\infty).$
\end{lemma}
\begin{proof}
As in Lemma~3.14 of \cite{Zw:12}, we simply integrate by parts using
the operator
$$
L=\frac h{iw_h} \frac{1}{\smallabs{\pa \varphi}^2} \pa \varphi \cdot \pa,
$$
chosen so that
$$
L^k e^{i\varphi w_h/h}= e^{i\varphi w_h/h}.
$$
The integration by parts then gains $(h/w_h)^k=O(h^k).$
\end{proof}

Thus as in the usual case, we may (decomposing $a$ using a partition of unity) read off stationary phase
asymptotics as a sum of asymptotics associated to each critical point,
for a nondegenerate $\varphi;$ also we may use the Morse Lemma to
convert $\varphi$ into a diagonal quadratic form near each of those
critical points.  The only difficulty is then to compute quadratic
stationary phase asymptotics, as in Theorem~3.13 of \cite{Zw:12}.

\begin{lemma}\label{lemma:quadphase}
Let $w_h \in \Omega_\ep.$
Let $$\varphi(x)=\frac 12 \smallang{Qx,x}$$ be a quadratic phase, with
$Q$ a nonsingular, symmetric, real matrix.
For all $N \in \NN,$
$$
I_h(a,\varphi; w_h)=\bigg(\frac{2 \pi h}{w_h}\bigg)^{\frac n2} \frac{e^{\frac{i\pi}4 \sgn Q}
}{\abs{\det Q}^{1/2}} \bigg( \sum_{k=0}^{N-1} \frac{1}{k!}\bigg(\frac
h{w_h} \bigg)^{k}
\bigg( \frac{ \ang{Q^{-1}D, D}}{2i}\bigg)^k a(0) + O(h^N) \bigg).
$$
\end{lemma}
\begin{proof}
Let $\ta$ denote an almost analytic extension of $a$ with support
in a small neighborhood (in $\CC^n$) of $\supp a$ (see \cite[Section
3.1]{Hormander1}, as well as, for instance, \cite[Chapter
8]{Dimassi-Sjostrand}).  Then 
\begin{equation}\begin{aligned}
I_h(a,\varphi; w_h)&= \int_{\RR^n} e^{i\varphi(x) w_h/h} \, a(x)\,
dx_1 \wedge \dots \wedge d x_n\\
&= \int_{\RR^n} e^{i\varphi(x\sqrt{w_h})/h} \, a(x)\,
dx_1 \wedge \dots \wedge d x_n\\
&= \int_{\Gamma} e^{i\varphi(z)/h} \, \ta(z/\sqrt{w_h})\, w_h^{-n/2} \, dz_1 \wedge \dots \wedge d z_n\,
\end{aligned}
\end{equation}
where $\Gamma$ is the complex contour $\{z_j =\sqrt{w_h} x_j,\ x_j \in \RR\}.$
Since $\ta$ is compactly supported, we may apply Stokes's theorem on
the domain
$$
\Upsilon =\big\{z \in \CC^n: z_j=((1-s) + s\sqrt{w_h}) x_j,\ x_j \in
\RR,\ s \in [0, 1]\big\}
$$
to obtain
\begin{multline}
  \label{eq:3}
I_h(a,\varphi; w_h)= \int_{\RR^n} e^{i\varphi(x)/h} \,
\ta(x/\sqrt{w_h})\, w_h^{-n/2} \, dx_1 \wedge \dots \wedge d x_n\\+
\iint_{\Upsilon} e^{i\varphi(z)/h} \,
\overline{\pa}\big[ \ta(z/\sqrt{w_h})\, w_h^{-n/2} \, dz_1 \wedge \dots \wedge d z_n\big].
\end{multline}

By almost-analyticity of $\ta,$ the latter integral is $O(h^\infty)$ since the support of the
integrand is compact and over this compact set $\Im (z/\sqrt{w_h})=O(\hlh)$ for $z \in
\Upsilon.$  The former integral is then an
ordinary stationary phase integral with quadratic phase to which we
may directly apply Theorem~3.13 of \cite{Zw:12}.  Note of course that
in applying the usual formula for the Gaussian integral, we are using
the fact that
\begin{equation}
\begin{aligned}
\frac{\pa^{\smallabs{\alpha}}}{\pa x^\alpha}\big( \ta(x/\sqrt{w_h})\big)\big\rvert_{x=0}
&=w_h^{-\abs{\alpha}/2} \frac{\pa^{\smallabs{\alpha}}}{\pa
  z^\alpha}\ta  (0)\\
&=w_h^{-\abs{\alpha}/2} \frac{\pa^{\smallabs{\alpha}}}{\pa
  x^\alpha}a  (0).
\end{aligned}
\end{equation}
\end{proof}

Assembling the foregoing results, we finally arrive at the desired
stationary phase expansion in general.  (Cf.\ Theorem~3.16 of \cite{Zw:12}.)
\begin{prop}\label{proposition:statphase}
Let $w_h \in \Omega_\ep,$ $a \in \CcI(\RR^n).$  Suppose $x_0$ is the
unique point in $\supp a$  at which $\pa
\varphi=0$ and that $\det \pa^2
\varphi(x_0)\neq 0.$  Then there exist differential operators
$A_{2k}(x,D)$ of order $\leq 2k$ such that for all $N \in \NN$
$$
I_h(a,\varphi; w_h) =\bigg( \sum_{k=0}^{N-1} A_{2k}(x,D)
  a(x_0) \bigg(\frac{h}{w_h}\bigg)^{k+n/2}\bigg) e^{\frac{i
      \varphi(x_0)w_h}{h}} +O(h^{N+n/2}).
$$
In particular, $$A_0=(2\pi)^{n/2} \abs{\det \pa^2 \varphi(x_0)}^{-1/2}
e^{\frac{i\pi}4 \sgn \pa^2 \varphi(x_0)}.$$
\end{prop}

\bibliographystyle{plain}
\bibliography{all}

\def\cprime{$'$} \def\cprime{$'$} \def\cprime{$'$}
  \def\polhk#1{\setbox0=\hbox{#1}{\ooalign{\hidewidth
  \lower1.5ex\hbox{`}\hidewidth\crcr\unhbox0}}} \def\cprime{$'$}
  \def\cprime{$'$} \def\cprime{$'$} \def\cprime{$'$} \def\cprime{$'$}
  \def\cprime{$'$} \def\cprime{$'$} \def\cprime{$'$} \def\cprime{$'$}
  \def\cprime{$'$} \def\cprime{$'$} \def\cprime{$'$} \def\cprime{$'$}
  \def\cprime{$'$} \def\cprime{$'$} \def\cprime{$'$} \def\cprime{$'$}
  \def\cprime{$'$} \def\cprime{$'$} \def\cprime{$'$} \def\cprime{$'$}
  \def\cprime{$'$} \def\cprime{$'$} \def\cprime{$'$} \def\cprime{$'$}
  \def\cprime{$'$} \def\cprime{$'$} \def\cprime{$'$}
\begin{bibdiv}
\begin{biblist}

\bib{Bardos-Guillot-Ralston1}{article}{
      author={Bardos, C.},
      author={Guillot, J.-C.},
      author={Ralston, J.},
       title={La relation de {P}oisson pour l'\'equation des ondes dans un
  ouvert non born\'e},
        date={1982},
     journal={Comm P.D.E.},
      volume={7},
       pages={905\ndash 958},
}

\bib{BaVaWu:15}{article}{
      author={Baskin, Dean},
      author={Vasy, Andr{\'a}s},
      author={Wunsch, Jared},
       title={Asymptotics of radiation fields in asymptotically {M}inkowski
  space},
        date={2015},
        ISSN={0002-9327},
     journal={Amer. J. Math.},
      volume={137},
      number={5},
       pages={1293\ndash 1364},
      review={\MR{3405869}},
}

\bib{BaWu:13}{article}{
      author={Baskin, Dean},
      author={Wunsch, Jared},
       title={Resolvent estimates and local decay of waves on conic manifolds},
        date={2013},
        ISSN={0022-040X},
     journal={J. Differential Geom.},
      volume={95},
      number={2},
       pages={183\ndash 214},
         url={http://projecteuclid.org/euclid.jdg/1376053445},
      review={\MR{3128982}},
}

\bib{BoFuRaZe:16}{misc}{
      author={Bony, Jean-Francois},
      author={Fujiie, Setsuro},
      author={Ramond, Thierry},
      author={Zerzeri, Maher},
       title={Resonances for homoclinic trapped sets},
        date={2016},
}

\bib{Burq:Coin}{article}{
      author={Burq, Nicolas},
       title={P\^oles de diffusion engendr\'es par un coin},
        date={1997},
        ISSN={0303-1179},
     journal={Ast\'erisque},
      number={242},
       pages={ii+122},
      review={\MR{1600338 (98m:35152)}},
}

\bib{Cheeger-Taylor1}{article}{
      author={Cheeger, J.},
      author={Taylor, M.E.},
       title={On the diffraction of waves by conical singularities. {I}},
        date={1982},
        ISSN={0010-3640},
     journal={Comm. Pure Appl. Math.},
      volume={35},
      number={3},
       pages={275\ndash 331},
        note={MR84h:35091a},
}

\bib{Cheeger-Taylor2}{article}{
      author={Cheeger, J.},
      author={Taylor, M.E.},
       title={On the diffraction of waves by conical singularities. {I}{I}},
        date={1982},
        ISSN={0010-3640},
     journal={Comm. Pure Appl. Math.},
      volume={35},
      number={4},
       pages={487\ndash 529},
        note={MR84h:35091b},
}

\bib{Ch:11}{article}{
      author={Christianson, Hans},
       title={Quantum monodromy and nonconcentration near a closed
  semi-hyperbolic orbit},
        date={2011},
     journal={Transactions of the American Mathematical Society},
      volume={363},
      number={7},
       pages={3373\ndash 3438},
}

\bib{Dimassi-Sjostrand}{book}{
      author={Dimassi, Mouez},
      author={Sj{\"o}strand, Johannes},
       title={Spectral asymptotics in the semi-classical limit},
      series={London Mathematical Society Lecture Note Series},
   publisher={Cambridge University Press},
     address={Cambridge},
        date={1999},
      volume={268},
        ISBN={0-521-66544-2},
      review={\MR{MR1735654 (2001b:35237)}},
}

\bib{Dy:16}{article}{
      author={Dyatlov, Semyon},
       title={Spectral gaps for normally hyperbolic trapping},
        date={2016},
     journal={Ann. Inst. Fourier, Grenoble},
      volume={66},
      number={1},
       pages={55\ndash 82},
}

\bib{DyZw:15}{article}{
      author={Dyatlov, Semyon},
      author={Zworski, Maciej},
       title={Mathematical theory of scattering resonances},
        date={2015},
     journal={Book in progress, http://math. mit. edu/dyatlov/res/(22 December
  2015, date last accessed)},
}

\bib{FoWu:17}{article}{
      author={Ford, G~Austin},
      author={Wunsch, Jared},
       title={The diffractive wave trace on manifolds with conic
  singularities},
        date={2017},
     journal={Advances in Mathematics},
      volume={304},
       pages={1330\ndash 1385},
}

\bib{Ga:14}{article}{
      author={Galkowski, Jeffrey},
       title={Distribution of resonances in scattering by thin barriers},
        date={2014},
     journal={arXiv preprint arXiv:1404.3709},
}

\bib{Ga:15}{article}{
      author={Galkowski, Jeffrey},
       title={A quantitative {V}ainberg method for black box scattering},
        date={2017},
     journal={Communications in Mathematical Physics},
      volume={349},
      number={2},
       pages={527\ndash 549},
}

\bib{Hillairet:2005}{article}{
      author={Hillairet, Luc},
       title={Contribution of periodic diffractive geodesics},
        date={2005},
        ISSN={0022-1236},
     journal={J. Funct. Anal.},
      volume={226},
      number={1},
       pages={48\ndash 89},
         url={http://dx.doi.org/10.1016/j.jfa.2005.04.013},
      review={\MR{2158175 (2006d:58026)}},
}

\bib{Hormander1}{article}{
      author={H{\"o}rmander, L.},
       title={Fourier integral operators, {I}},
        date={1971},
     journal={Acta Math.},
      volume={127},
       pages={79\ndash 183},
}

\bib{Hormander3}{book}{
      author={H{\"o}rmander, L.},
       title={The analysis of linear partial differential operators},
   publisher={Springer-Verlag},
     address={Berlin{,} Heidelberg{,} New York{,} Tokyo},
        date={1985},
      volume={3},
}

\bib{Lax-Phillips1}{book}{
      author={Lax, P.D.},
      author={Phillips, R.S.},
       title={Scattering theory},
   publisher={Academic Press},
     address={New York},
        date={1967},
        note={Revised edition, 1989},
}

\bib{MR83j:35128}{article}{
      author={Melrose, Richard},
       title={Scattering theory and the trace of the wave group},
        date={1982},
        ISSN={0022-1236},
     journal={J. Funct. Anal.},
      volume={45},
      number={1},
       pages={29\ndash 40},
      review={\MR{83j:35128}},
}

\bib{Melrose-Wunsch1}{article}{
      author={Melrose, Richard},
      author={Wunsch, Jared},
       title={Propagation of singularities for the wave equation on conic
  manifolds},
        date={2004},
        ISSN={0020-9910},
     journal={Invent. Math.},
      volume={156},
      number={2},
       pages={235\ndash 299},
      review={\MR{MR2052609 (2005e:58048)}},
}

\bib{Nonn:12}{incollection}{
      author={Nonnenmacher, St\'ephane},
       title={Quantum transfer operators and quantum scattering},
        date={2012},
   booktitle={Seminaire: {E}quations aux {D}\'eriv\'ees {P}artielles.
  2009--2010},
      series={S\'emin. \'Equ. D\'eriv. Partielles},
   publisher={\'Ecole Polytech., Palaiseau},
       pages={Exp. No. VII, 18},
}

\bib{NSZ:11}{article}{
      author={Nonnenmacher, St\'ephane},
      author={Sj\"ostrand, Johannes},
      author={Zworski, Maciej},
       title={From open quantum systems to open quantum maps},
        date={2011},
     journal={Comm. Math. Phys.},
      volume={304},
      number={1},
       pages={1\ndash 48},
}

\bib{Sjostrand-Zworski1}{article}{
      author={Sj{\"o}strand, J.},
      author={Zworski, M.},
       title={Complex scaling and the distribution of scattering poles},
        date={1991},
     journal={J. Amer. Math. Soc.},
      volume={4},
       pages={729\ndash 769},
}

\bib{Sjostrand-Zworski3}{article}{
      author={Sj{\"o}strand, J.},
      author={Zworski, M.},
       title={Lower bounds on the number of scattering poles},
        date={1993},
     journal={Comm. P.D.E.},
      volume={18},
       pages={847\ndash 858},
}

\bib{Sjostrand-Zworski5}{article}{
      author={Sjostrand, Johannes},
      author={Zworski, Maciej},
       title={Lower bounds on the number of scattering poles, ii},
        date={1994},
     journal={Journal of Functional Analysis},
      volume={123},
      number={2},
       pages={336\ndash 367},
}

\bib{Sjostrand-Zworski:Fractal}{article}{
      author={Sj{\"o}strand, Johannes},
      author={Zworski, Maciej},
       title={Fractal upper bounds on the density of semiclassical resonances},
        date={2007},
        ISSN={0012-7094},
     journal={Duke Math. J.},
      volume={137},
      number={3},
       pages={381\ndash 459},
      review={\MR{MR2309150}},
}

\bib{Vainberg:Exterior}{article}{
      author={Va{\u\i}nberg, B.~R.},
       title={Exterior elliptic problems that depend polynomially on the
  spectral parameter, and the asymptotic behavior for large values of the time
  of the solutions of nonstationary problems},
        date={1973},
     journal={Mat. Sb. (N.S.)},
      volume={92(134)},
       pages={224\ndash 241, 343},
      review={\MR{0346319 (49 \#11044)}},
}

\bib{Vainberg:Asymptotic}{book}{
      author={Va{\u\i}nberg, B.~R.},
       title={Asymptotic methods in equations of mathematical physics},
   publisher={Gordon \& Breach Science Publishers},
     address={New York},
        date={1989},
        ISBN={2-88124-664-8},
        note={Translated from the Russian by E. Primrose},
      review={\MR{1054376 (91h:35081)}},
}

\bib{Wu:16}{article}{
      author={Wunsch, Jared},
       title={Diffractive propagation on conic manifolds},
        date={2016},
     journal={arXiv preprint arXiv:1605.00502},
}

\bib{Zw:99}{article}{
      author={Zworski, Maciej},
       title={Poisson formula for resonances in even dimensions},
        date={1998},
        ISSN={1093-6106},
     journal={Asian J. Math.},
      volume={2},
      number={3},
       pages={609\ndash 617},
         url={http://dx.doi.org/10.4310/AJM.1998.v2.n3.a6},
      review={\MR{1724627}},
}

\bib{Zw:12}{book}{
      author={Zworski, Maciej},
       title={Semiclassical analysis},
   publisher={American Mathematical Soc.},
        date={2012},
      volume={138},
}

\end{biblist}
\end{bibdiv}

\end{document}